\documentclass[oneside,a4paper,12pt,centertags]{amsart} 
\usepackage{amssymb,amsthm,amscd,latexsym,amsmath,verbatim,color,vmargin}
\usepackage[all]{xy}
\usepackage[bookmarks=true]{hyperref}

\theoremstyle{plain}
\newtheorem{Thm}{Theorem}[section]

\theoremstyle{definition}
\newtheorem{Prop}[Thm]{Proposition}
\newtheorem{Lem}[Thm]{Lemma}

\newtheorem{Def}[Thm]{Definition}

\newtheorem{Eg}[Thm]{Example}

\theoremstyle{remark}
\newtheorem{Rem}[Thm]{Remark}

\newcommand{\R}{\mathbb{R}}
\newcommand{\C}{\mathbb{C}}
\newcommand{\N}{\mathbb{N}}

\newcommand{\re}{\mathop{\mathrm{Re}}}

\newcommand{\Dom}{\textsf{Dom}}
\newcommand{\Ran}{\textsf{Ran}}
\newcommand{\Nul}{\textsf{Nul}}

\newcommand{\sgn}{\mathop{\text{sgn}}}

\newcommand\ca{1}
\renewcommand{\div}{\mathop{\mathrm{div}}\nolimits}
\newcommand{\grad}{\nabla}
\newcommand{\sppt}{\mathop{\mathrm{sppt}}\nolimits}
\newcommand{\loc}{\mathrm{loc}}
\newcommand{\inj}{\mathop{\mathrm{inj}}}
\newcommand{\Ric}{\mathop{\mathrm{Ric}}}
\newcommand{\End}{\mathop{\mathrm{End}}}

\newcommand{\clos}[1]{\overline{#1}}

\newcommand{\bdy}{\partial}
\newcommand{\ta}{{\scriptscriptstyle \parallel}}
\newcommand{\no}{{\scriptscriptstyle\perp}}
\mathchardef\semic="303B
\newcommand{\sett}[2]{ \{ #1 \, \semic \, #2 \} }

\mathchardef\semic="303A

\newcommand{\mH}{{\mathcal H}}

\newcommand{\mX}{{\mathcal X}}
\newcommand{\mY}{{\mathcal Y}}

\newcommand{\mE}{{\mathcal E}}

\newcommand{\scl}[2]{\langle #1,#2 \rangle}

\newcommand{\supp}{\text{{\rm supp}}\,}
\newcommand{\dist}{\text{{\rm dist}}\,}

\newcommand{\dom}{\textsf{Dom}}
\newcommand{\ran}{\textsf{Ran}}
\newcommand{\nul}{\textsf{Nul}}

\newcommand{\barint}{\mbox{$ave \int$}}
\newcommand{\divv}{{\text{{\rm div}}}}

\newcommand{\esssup}{\text{{\rm ess sup}}}

\newcommand{\pd}{\partial}

\newcommand{\rev}[1]{\overline{#1}}

\newcommand{\mV}{\mathcal{V}}

\def\barint_#1{\mathchoice
            {\mathop{\vrule width 6pt
height 3 pt depth -2.5pt
                    \kern -8.8pt
\intop}\nolimits_{#1}}%
            {\mathop{\vrule width 5pt height
3 pt depth -2.6pt
                    \kern -6.5pt
\intop}\nolimits_{#1}}%
            {\mathop{\vrule width 5pt height
3 pt depth -2.6pt
                    \kern -6pt
\intop}\nolimits_{#1}}%
            {\mathop{\vrule width 5pt height
3 pt depth -2.6pt
          \kern -6pt \intop}\nolimits_{#1}}}

\allowdisplaybreaks

\numberwithin{equation}{section}

\begin{document}

\title[]{Quadratic estimates for degenerate elliptic systems on manifolds with lower Ricci curvature bounds and boundary value problems }
\author[Auscher]{Pascal Auscher}
\author[Morris]{Andrew J. Morris$\,^1$}
\author[Ros\'{e}n]{Andreas Ros\'{e}n}

\address{Pascal Auscher\\ Universit\'e Paris-Saclay, CNRS, Laboratoire de Math\'{e}matiques d'Orsay, 91405 Orsay, France}
\email{pascal.auscher@universite-paris-saclay.fr}
\address{Andrew J. Morris\\ School of Mathematics\\University of Birmingham\\Edgbaston\\Birmingham\\B15 2TT\\UK}
\email{a.morris.2@bham.ac.uk}
\address{Andreas Ros\'en\\Mathematical Sciences, Chalmers University of Technology and the University of Gothenburg\\
SE-412 96 G{\"o}teborg, Sweden}
\email{andreas.rosen@chalmers.se}

\date{\today}
\thanks{$\,^1$This work was supported by the Engineering and Physical Sciences Research Council [grant number EP/J010723/1]. No data were created or analysed in this study.}

\begin{abstract}
Weighted quadratic estimates are proved for certain bisectorial 
first-order differential operators with bounded measurable coefficients
which are (not necessarily pointwise) accretive,
on complete manifolds with positive injectivity radius.   
As compared to earlier results, Ricci curvature is only assumed 
to be bounded from below, and the weight is only assumed
to be locally in $A_2$.
The Kato square root estimate is proved under this weaker 
assumption. 
On compact Lipschitz manifolds we prove solvability estimates for solutions to 
degenerate elliptic systems with not necessarily self-adjoint coefficients, 
and with Dirichlet, Neumann and Atiyah--Patodi--Singer boundary conditions.
\end{abstract}

\maketitle

\tableofcontents

\section{Introduction}

This paper concerns the weighted $L^2$ theory for generalized singular
integrals appearing in the $H^\infty$ functional calculus
of differential operators with non-smooth coefficients. 
The seminal work in this area is the  
Kato square root estimate by Auscher, Hofmann, Lacey, McIntosh and Tchamitchian~\cite{AHLMcT:02}, which  yields estimates
\begin{equation}   \label{eq:KatoRiesz}
   \|\nabla(-\divv A\nabla)^{-1/2}u\|_2\eqsim \|u\|_2
\end{equation}
of Riesz transforms 
associated with divergence form operators $\divv A \nabla$ on $\R^n$.
The coefficient matrix $A$ is assumed to be merely bounded, 
measurable and accretive,
complex matrix-valued.
Although the square root in \eqref{eq:KatoRiesz} is in the  holomorphic  functional calculus of the sectorial operator $-\divv A\nabla$, the left factor $\nabla$ is not.
The Kato estimate was subsequently put into its natural functional analytic framework, and boundedness of the $H^\infty$ functional calculus for 
certain bisectorial Dirac type operators was shown.
First, generalizations $\Gamma+B^{-1}\Gamma^* B$ of the Hodge--Dirac operator were
considered in \cite{AKMc}, and later perturbations $DB$ of self-adjoint differential
operators $D$ by accretive multipliers $B$ were considered
in \cite{AAMc2, AAMc1} and in later works on boundary value problems
discussed below.
The $DB$ operators have a simpler structure than Hodge--Dirac operators, 
but still the bounds for $\Gamma+B^{-1}\Gamma^* B$ operators follow
from the bounds for DB operators, using Hodge splittings as shown in \cite[Sec.~10.1]{AAMc1}. Also the Kato square root estimate
\eqref{eq:KatoRiesz} is a well-known corollary,  see \cite[Sec.~2.1]{AAMc1}. 
The main ideas in all the proofs are the following.
(1) The $H^\infty$ functional calculus bounds are reduced to proving quadratic estimates, which provide a Littlewood--Paley decomposition of $L^2$
adapted to the operator. (2) A local $Tb$ theorem allows
a further reduction of the problem to showing Carleson bounds for a certain
multiplier. (3) This Carleson estimate is established using a stopping
time argument. 

The $L^2$ estimates for the above operators are quite sharp in the sense that 
$L^p$ estimates in general only hold in a small interval around $p=2$ 
for the above operators, and not for all $1<p<\infty$ as is the case for 
classical Calder\'on--Zygmund operators.
See \cite{Auscher:07, AuscherStahlhut:16, FreyMcIntoshPortal:18}. 
For $L^2$ spaces, the bounds for the $H^\infty$ functional calculus of the 
above Dirac type operators have been extended (1) to
certain complete manifolds $M$ in \cite{MM,BM}, 
with earlier results for compact manifolds in \cite{AKMc}, 
and (2) to weighted $L^2(\R^n,w)$ spaces in \cite{ARR}. 
The main issue on manifolds, which are assumed to be complete but are
allowed to 
be noncompact, is the geometry of the manifold towards infinity.
The main issue in weighted spaces is local and the weight is assumed
to be in the Muckenhoupt $A_2$ class. The bounds of the 
$H^\infty$ functional calculus here require the double stopping time argument
from \cite[Sec.~3.5-7]{ARR}.

  A first main result in this paper is the following theorem,
 which proves quadratic estimates for operators $DB$ under weaker geometric assumptions than those in \cite{AKMc, MM, BM}, and which combines those results with the weighted theory developed in \cite{ARR}.  
See Section~\ref{sec:2} for precise definitions.
 We adopt the convention for estimating $x,y\geq0$ whereby $x\lesssim y$ means that there exists a constant $C< \infty$, which only depends on constants specified in the relevant preceding hypotheses, such that $x\leq Cy$. We write $x\eqsim y$ when $x\lesssim y \lesssim x$. 

\begin{Thm}    \label{thm:main}
Let $(M,g)$ denote a complete  Riemannian manifold with injectivity radius $\inj(M,g)\geq r_E$ and Ricci curvature $\Ric(M,g) \geq -K g$ in the sense of bilinear forms,
where $r_E>0$ and $K<\infty$. 
Let $\omega \in A_2^R(M)$ be a local Muckenhoupt  weight
as in Section~\ref{sec:weightsops}, for some $R>0$.
  Consider perturbations $DB$ as in Section~\ref{sec:weightsops},
of the self-adjoint operator $D$ in $L^2(\mathcal{V},\omega)$ as given by \eqref{eq:Vdefn}-\eqref{eq:Ddefn}, by bounded coefficients $B$ as in \eqref{eq:Bsplit}, which are accretive in
the sense of \eqref{eq:acc}.  
Then we have quadratic estimates
\begin{equation}\label{eq:mainest}
\int_0^\infty \|tDB(I+t^2(DB)^2)^{-1}u\|_{L^2(\mathcal{V},\omega)}^2 \frac{dt}{t} \lesssim \|u\|_{L^2(\mathcal{V},\omega)}^2, \qquad u \in L^2(\mathcal{V},\omega),
\end{equation}
where the implicit constant depends only on $n=\dim(M)$, $r_E$, $K$, 
$\|B\|_\infty$, $\kappa_B$, $R$ and 
$[\omega]_{A_2^R(M)}$.
\end{Thm}

 There are several novelties in Theorem~\ref{thm:main}, and its proof, as compared 
to earlier results. The best previous result \cite[Thm.~1.1]{BM}, for complete manifolds 
in the unweighted setting $\omega=1$, required an additional upper bound on the Ricci curvature.
We avoid this by using a result of Anderson and Cheeger~\cite{AC} to observe that the generalised bounded geometry 
property introduced by Bandara and McIntosh~\cite[Defn.~2.5]{BM}, see \eqref{eq:hr0} below and the related discussion, holds on the tangent bundle $TM$ 
whenever $M$ has a positive injectivity radius and only a lower bound on its Ricci curvature. 

Our proof only uses the injectivity radius and Ricci curvature hypotheses to obtain coordinates which satisfy \eqref{eq:hr0}, 
so Theorem~\ref{thm:main} actually holds more generally under the assumption that $TM$ has such generalised bounded geometry.  
It is this property which allows us to localize the quadratic estimate to coordinate charts.
For general local $A_2$ weights $\omega$, we pull back the 
estimates to $\R^n$ in order to apply the Euclidean 
weighted quadratic estimate from \cite{ARR}.
In this regard, our approach is more aligned with the earlier treatment for compact manifolds in \cite[Sec.~7]{AKMc} than with those for complete manfiolds in \cite{BM,MM}, where estimates are established directly on the manifold.

In order to apply the weighted quadratic estimates from \cite{ARR}, we need to extend the weight $\omega$, as well as the coefficients $B$, from within a euclidean ball to all
of $\R^n$, whilst preserving the $A_2$ property for $\omega$ and the accretivity for $B$. This technical issue is resolved in Section~\ref{sec:extension}.
For our estimates to apply to general elliptic systems,
we do not assume pointwise accretivity of the coefficients, but only
a G\aa rding inequality, which is what makes the extension problem
for the coefficients non-trivial. 

Theorem~\ref{thm:main} yields in particular the following 
extension of 
the Kato square root estimate, which we prove at the end of Section~\ref{sect:loc}.

\begin{Thm}  \label{thm:Kato}
Let $(M,g)$ denote a complete Riemannian manifold with $\inj(M,g)\ge r_E$ and 
$\Ric(M,g) \geq -K g$ in the sense of bilinear forms, for some $r_E>0$ and $K<\infty$. Let $\omega \in A_2^R(M)$, for some $R>0$.
Let
$$
  J(u,v)= \int_M \Big( \scl{A\nabla_M u}{\nabla_M v}+ u \scl{b}{ \nabla_M v}+ \scl{c}{ \nabla_M \overline u}\overline v+ du\overline v \Big) \omega(x) d\mu(x)
$$
be a sesqui-linear form
on $L^2(M,\omega)$, 
with domain $W^{1,2}(M,\omega)$  and with $A\in L^\infty(\End(T M))$, $b,c\in L^\infty(T M)$ and
$d\in L^\infty(M)$, which satisfies
 a G\aa rding inequality  
$$
  \re J(u,u)\gtrsim \|\nabla_M u\|_{L^2(M,\omega)}^2+ \|u\|_{L^2(M,\omega)}^2,
$$
for all $u\in W^{1,2}(M,\omega)$.
Consider the associated operator 
$$
L= -\divv_{M,\omega} A\nabla_M-\divv_{M,\omega} b+ c\nabla_M+d
$$ 
and a function 
$a\in L^\infty(M)$ with $\inf_M \re a>0$.
Then the weighted Kato square root estimate
$$
  \| \sqrt{aL}u \|_{L^2(M,\omega)}\eqsim \|\nabla_M u\|_{L^2(M,\omega)}+ \|u\|_{L^2(M,\omega)}, \qquad u\in W^{1,2}(M,\omega),
$$
holds on $M$.
\end{Thm}

The novelty in Theorem~\ref{thm:Kato} is that only a lower bound
on the 
Ricci curvature of the manifold is assumed and the 
weight is only assumed to be a local $A_2$ weight.
The proof involves applying
Theorem~\ref{thm:main} to suitable coefficients $B$ depending on 
$A, a, b, c, d$.

An early application of the boundedness of the $H^\infty$
functional calculus of $DB$ operators was to 
elliptic non-smooth boundary value
problems (BVPs) \cite{AAH, AAMc2, AA1, AR2}.
As in the theory of smooth Dirac BVPs, see for example
\cite{BossBavenbeck}, the key idea is to write the derivative transversal 
to the boundary in the partial differential equation (PDE) in terms of an operator acting along the
boundary. For non-smooth divergence form elliptic equations
$\divv A\nabla u=0$, it was found in
\cite{AAMc2}
that the structure of this boundary operator is precisely $DB$.
This uses a non-linear transform $A\mapsto B$ of accretive coefficients,
featuring a Schur complement of the transversal block in $A$, 
see \eqref{eq:defntranformedB}. 
Theorem~\ref{thm:main} applies to prove
boundedness of projections onto Hardy type subspaces in particular,
which is crucial for BVPs.

A second main result in this paper is the following theorem,
which uses Theorem~\ref{thm:main} to estimate the Neumann and Dirichlet
boundary data, as a whole, for solutions to divergence form elliptic equations.

\begin{Thm}   \label{thm:ND} 
Let $\Omega$ be a compact manifold with Lipschitz boundary 
$\bdy\Omega$,  in the sense that
we have a bilipschitz parametrization $\rho(t,x)$  of a neighbourhood, 
relative $\Omega$,  of $\partial \Omega$
by a cylinder $[0,\delta)\times M$ as in \eqref{eq:paramcyl}, with 
$M$ being a closed Riemannian manifold.
Let $\omega\in A_2(\Omega)$ be a Muckenhoupt weight such that 
$\omega_\rho= \omega\circ\rho$ is independent of
$t$, and let $\omega_0= \omega_\rho|_M$.
Consider weak solutions $u$ to a divergence form elliptic system
$\divv A\nabla u=0$ in $\Omega$, with degenerate elliptic coefficients
$A$ satisfying $A/\omega\in L^\infty(\End(T\Omega))$ and the
accretivity conditions \eqref{eq:intaccr}, \eqref{eq:bdyaccr}.
Assume that a boundary trace $A_0=\lim_{t\to 0} A_\rho$
exists for the pulled back coefficients $A_\rho$  
on the cylinder, in the 
sense that we have a finite Carleson discrepancy    
$\|A_\rho/\omega_\rho-A_0/\omega_0\|_*<\infty$. See \eqref{eq:modifiedCarl}.
Consider the conormal derivative $\pd_{\nu_{A_\rho}}u_\rho$
and the tangential gradient $\nabla_M u_\rho$ of $u_\rho= u\circ \rho$
 on the cylinder. See 
\eqref{eq:conormalgrad}.

{\rm (i)} 
If $\|\nabla u\|_\mX<\infty$, 
with the modified non-tangential maximal function norm 
$\|\cdot\|_\mX$ given by \eqref{eq:Xnormdef}, then  
\begin{equation}   \label{eq:Xest0}
   \|\pd_{\nu_{A_\rho}}u_\rho|_M\|_{L^2(M,\omega_0^{-1})}
   + \|\nabla_M u_\rho|_M\|_{L^2(TM,\omega_0)}
   \lesssim \|\nabla u\|_\mX,
\end{equation}
where the traces exist in the
$L^2$ average sense \eqref{eq:Diniconv}.

{\rm (ii)} 
If $\|\nabla u\|_\mY<\infty$, with the square function norm $\|\cdot\|_\mY$ given by
\eqref{eq:Ynormdef}, then  
\begin{equation}   \label{eq:Yest0}
   \|\pd_{\nu_{A_\rho}}u_\rho|_M\|_{W^{-1,2}(M,\omega_0^{-1})}
   + \|\nabla_M u_\rho|_M\|_{W^{-1,2}(TM,\omega_0)}
   \lesssim \|\nabla u\|_\mY,
\end{equation}
where the traces exist in the
$W^{-1,2}$ sense \eqref{eq:Wminoneconv}.

Moreover, if $\|A_\rho/\omega_\rho-A_0/\omega_0\|_*$ is small enough, 
with smallness
depending on $[\omega_0]_{A_2(M)}$, $\|A_0/\omega_0\|_\infty$ and the
constant in \eqref{eq:bdyaccr}, then
the estimates $\gtrsim$ also hold in \eqref{eq:Xest0} and \eqref{eq:Yest0},
with the implicit constant also depending on the constant in 
\eqref{eq:intaccr}.
\end{Thm}

We remark that the assumed smoothness of the manifold $M$ 
is not used quantitatively, and the manifold $\Omega$ may be merely a Lipschitz manifold. 
Note also that the small Carleson condition on 
$A_\rho/\omega_\rho-A_0/\omega_0$ required for the reverse 
estimates $\gtrsim$, in general is only satisfied for pull backs
of smooth coeficients $A$ if the Lipschitz character of $\bdy\Omega$
is small, unless this boundary is locally the graph of a Lipschitz
function. 
The above Neumann and Dirichlet data for $u_\rho$ correspond 
to those for $u$ on $\bdy \Omega$ as in \eqref{eq:transfofCdata}.
We nevertheless use the above formulation since pullback by $\rho$
does not preserve the Sobolev space $W^{-1,2}$.

Theorem~\ref{thm:ND} is essentially a solvability result for the Atiyah--Patodi--Singer (APS) boundary value problem associated with the PDE $\div A\nabla u=0$. The connection to the classical APS conditions is that we rewrite the second-order equation as a first-order system \eqref{eq:firstCR} for the (conormal) gradient of $u$. In the smooth setting, this first-order system would coincide with the Hodge--Dirac operator acting on gradient vector fields, and our APS conditions arise in the well-known way by considering spectral projections of an adapted operator on the boundary. See \eqref{eq:f0split} for details. Under suitable identifications our spectral projection $E_0^+$ would correspond to the Calder\'on projector for the original second-order equation.

We prove such APS solvability results in
Theorems~\ref{thm:N} and \ref{thm:D}  
in Section~\ref{sec:aps},  from  which
Theorem~\ref{thm:ND} is an immediate corollary.
Theorem~\ref{thm:ND} extends results from  \cite{AR2},
in that we both allow for domains $\Omega$ with arbitrary
topology and for $A_2$ degenerate equations.
Moreover, the Fredholm argument presented in Section~\ref{sec:L2trace}
is new; the proof from \cite[Sec.~16]{AR2} 
does not generalize to domains with arbitrary topology.
We limit ourselves to compact manifolds $M$ in Theorem~\ref{thm:ND},
for which the curvature  and injectivity radius hypotheses  in Theorem~\ref{thm:main} hold trivially.
This restriction is necessary for the Fredholm arguments involved in the proof,
which localize the Carleson regularity condition to the 
$\delta$-neighbourhood of $\bdy\Omega$. 
We remark that the dependence of the implicit constant in the estimates of 
$\|\nabla u\|_\mX$ and $\|\nabla u\|_\mY$
in Theorems~\ref{thm:ND}, \ref{thm:dirneu}, \ref{thm:N}
and \ref{thm:D} is hard to trace, as its finiteness is proved by an 
application of the Open Mapping Theorem. 

It is important to note that the APS solvability results discussed
above require only the Carleson regularity condition on the
coefficients near $\bdy \Omega$.
The reason is that the APS boundary condition is defined in
terms of operators in the $H^\infty$ functional calculus of the boundary
operator $DB$.
However, solution operators for the classical Neumann and Dirichlet BVPs, 
such as the
Dirichlet-to-Neumann map, are outside of this functional calculus, 
just like the Riesz transform
in \eqref{eq:KatoRiesz} is outside the $H^\infty$ functional calculus of
 $\divv A \nabla$.
As a consequence, $L^2$ solvability of the Neumann and Dirichlet BVPs
requires further assumptions on the coefficients, notably self-adjointness.
See \cite[Sec.~4]{AAMc2}.
We prove in Section~\ref{sec:dirneu} the following.

\begin{Thm}  \label{thm:dirneu}
   Let $\Omega$, $\rho$, $M$, $\omega$ and $A$ be as in 
   Theorem~\ref{thm:ND}, and assume that
   $\|A_\rho/\omega_\rho-A_0/\omega_0\|_*$ and
   $\|(A_0^*-A_0)/\omega_0\|_\infty$ are small enough, 
depending on $[\omega_0]_{A_2(M)}$, $\|A_0/\omega_0\|_\infty$ and the
constant in \eqref{eq:bdyaccr}.
Then Neumann and Dirichlet regularity solvability estimates
$$
\|\nabla u\|_\mX\lesssim\min\Big(\|\pd_{\nu_{A_0}}u_\rho|_M\|_{L^2(M,\omega_0^{-1})},
 \|\nabla_M u_\rho|_M\|_{L^2(TM,\omega_0)}+ \sum_i \Big|\int_{M_i}\pd_{\nu_{A_\rho}}u_\rho d\mu\Big|\Big)
$$
hold
for all weak solutions to $\div A\nabla u=0$ in $\Omega$ with
$\|\nabla u\|_\mX<\infty$, and Dirichlet solvability estimates
$$
\|\nabla u\|_\mY\lesssim  \|\nabla_M u_\rho|_M\|_{W^{-1,2}(TM,\omega_0)}
+ \sum_i \Big|\int_{M_i}\pd_{\nu_{A_\rho}}u_\rho d\mu\Big|
$$
hold
for all weak solutions to $\div A\nabla u=0$ in $\Omega$ with
$\|\nabla u\|_\mY<\infty$.
Here $M_i$ denote the connected components of $M$, and
the implicit constants in the estimates of $\|\nabla u\|_\mX$ and
$\|\nabla u\|_\mY$ also depend on the constant in 
\eqref{eq:intaccr}.
\end{Thm}

Our solvability estimates for the Neumann and Dirichlet
BVPs in Theorem~\ref{thm:dirneu} and for the Atiyah--Patodi--Singer 
BVP in Theorems~\ref{thm:N} and \ref{thm:D}  are new in that we go beyond 
self-adjoint coefficients, as well as allow for degenerate
coefficients. Our method relies on the boundedness of the
functional calculus of the boundary $DB$ operator, the 
construction of which depends on the Lipschitz parametrization $\rho$
of the boundary $\bdy \Omega$.
In the weighted $L^2$ and $W^{-1,2}$ spaces this
yields not only solvability estimates, but also analytical
dependence of the solution operators on the coefficients.
However, at this level of generality of the coefficents
there is little room for generalizations to other 
function spaces and boundaries rougher than Lipschitz,
as considered for example in \cite{Mitrea3Taylor:16}.

\subsection*{Acknowledgements}

The authors acknowledge support for several collaborative visits which enabled this work, from the Universit\'e Paris-Saclay, the Mathematical Institute at the University of Oxford, Chalmers University of Technology and the University of Gothenburg, and the Hausdorff Research Institute for Mathematics in Bonn. We thank David Rule for interesting mathematical discussions at the outset of this project. A special thanks to Lashi Bandara for his continued interest and patience while this manuscript was being completed.   Finally, we thank the anonymous referees and Gianmarco Brocchi for comments that improved the paper.

\section{Setup for quadratic estimates}   \label{sec:2}

This section aims at making the statement 
of Theorem~\ref{thm:main} precise. 
We begin with geometric preliminaries in order to define useful coverings for manifolds with a lower Ricci curvature bound and positive injectivity radius. The relevant weights and differential operators are then made precise in this context. A brief overview of the $H^\infty$ functional calculus for bisectorial operators is also included.

\subsection{Coverings} \label{sec:covering} 

Let $(M,g)$ denote a complete, not necessarily compact, Riemannian manifold $M$ with Riemannian metric $g$ and dimension 
$n=\dim(M)$, which is smooth in the sense that both the transition maps between coordinate charts on $M$, and the coefficients of the metric $g$ in coordinate charts, are infinitely differentiable.
We denote the geodesic distance by $d$ and the Riemannian measure by~$\mu$. A ball $B(x,r)$ in $M$ always refers to an open geodesic ball $B(x,r)=\{y\in M : d(x,y)<r\}$ whilst $V(x,r)=\mu(B(x,r))$ and $(\alpha B)(x,r)= B(x,\alpha r)$ for $\alpha>0$, $r>0$ and $x\in M$. For $E$, $F \subseteq M$, we set $d(E,F): = \inf \{d(E,F) : x\in E, y\in F\} $ and let $\ca_E$ denote the characteristic function of $E$.

We assume that the Ricci curvature on the manifold is bounded below in the sense that $\Ric(M,g)\geq -K g$, as bilinear forms, for some $K<\infty$. The results of Saloff-Coste~\cite[Thm.~5.6.4]{SC}, see also 
Hebey~\cite[Thm.~1.1]{H} and \cite[Lem.~1.1]{H}, then show that $M$ has exponential volume growth in the sense that 
\begin{equation} \label{eq:expgrowth}
0<V(x,\alpha r)\leq \alpha^n e^{\lambda \alpha r} V(x,r)<\infty,
\qquad \alpha\geq1,\ r>0,\ x\in M,
\end{equation}
where $\lambda= \sqrt{(n-1)K}$. It is also shown there that for any $0<T<\infty$, there exists a sequence $(x_i)_{i\in\N}$ of points in $M$ such that the balls $(B(x_i,T/2))_{i\in\N}$ are pairwise disjoint. Moreover, for each $\alpha \geq 1$, the collection $(B(x_i,\alpha T))_{i\in\N}$ is a locally finite covering of $M$ with the property that the number of balls in the collection intersecting any given ball from the collection is at most $N<\infty$, where $N$ depends only on $\alpha$, $T$, $n$ and $K$.

Our final geometric assumption is that the manifold also has a positive injective radius $\inj(M,g)\geq r_E>0$. Under this assumption, and the lower
bound on Ricci curvature, the existence of charts with $L^\infty$ control of the metric was proved by
Anderson and Cheeger~\cite[Thm.~0.3]{AC}. Their results show that for any $\beta\in(0,1)$ and any $C_{gbg}>1$, there exists $r_H>0$, depending only on $n$, $K$, $r_E$, $\beta$ and $C_{gbg}$, such that at each $x\in M$ there is a (harmonic) coordinate chart $\varphi:B(x,r_H)\rightarrow \R^n$ with a certain $C^{0,\beta}$-control of the metric, and which implies that 
\begin{equation}\label{eq:hr0}
C_{gbg}^{-1} \leq g(y) \leq C_{gbg}, \qquad y\in B(x,r_H),
\end{equation}
as a bilinear form. A typical advantage of working in such harmonic coordinates, as opposed to geodesic normal coordinates, is that Ricci curvature bounds can often supplant bounds on the full Riemann curvature tensor. This is ultimately what enables the bound \eqref{eq:hr0}. The precise formulation of this result can be found in the work by Hebey and Herzlich \cite[Thm.~6~\&~Cor.]{HH} (see also \cite[Thm.~1.3]{H:96}), which also recovered the $C^{1,\beta}$-estimates for the metric obtained by Anderson in \cite{An:90}, and established more general $C^{k,\beta}$-estimates, all of which require higher-order bounds on the Ricci curvature that are not assumed in this paper.

\subsection{Weights and operators}   \label{sec:weightsops}

A weight $\omega$ on $M$ refers to a non-negative measurable function $\omega:M\rightarrow[0,\infty]$. For $R>0$, let $A_2^R(M)$ denote the set of weights satisfying
\[
[\omega]_{A_2^R(M)} = \sup_{x\in M, r\leq R} \left(\barint_{B(x,r)} \omega \ d\mu \right) \left( \barint_{B(x,r)}  \frac{1}{\omega} \ d\mu \right) < \infty,
\]
where $\barint_B f d\mu=\mu(B)^{-1}\int_B f d\mu$ denotes the integral average with respect to the integral measure. Let
$A_2(M)$ denote the set of weights such that
$[\omega]_{A_2(M)}=\sup_{R>0}[\omega]_{A_2^R(M)}<\infty$.

  A vector bundle $\mathcal{W}$ on $M$ refers to a smooth complex vector bundle $\pi:\mathcal{W}\rightarrow M$ equipped with a smooth Hermitian metric $\langle\cdot,\cdot\rangle_{\mathcal{W}_x}$ on the fibres $\mathcal{W}_x$.
For a weight $\omega$ on $M$, the Hilbert space $L^2(\mathcal{W},\omega)$ of weighted square integrable sections of $\mathcal{W}$ has the inner product
\[
\langle u, v \rangle_{L^2(\mathcal{W},\omega)} = \int_M \langle u(x), v(x) \rangle_{\mathcal{W}_x} \ \omega(x)\ d\mu(x).
\]
A linear operator $T$ on $L^2(\mathcal{W},\omega)$ has a domain $\Dom(T)$, range $\Ran(T)$ and null space $\Nul(T)$ that are subspaces of $L^2(\mathcal{W},\omega)$.
 For the trivial bundle $M\times \C$, we often write $L^2(M,\omega)$ 
instead of $L^2(  M\times \C,\omega)$.  
We will consider in particular the complexified tangent bundle $TM$ and the endomorphism bundle $\End(TM)$, each equipped with the Hermitian metric induced by the Riemannian metric $g$ on $TM$. 
We identify the dual cotangent bundle $T^*M$ with $TM$
through the metric $g$.

We now introduce the first-order systems $DB$ considered in 
Theorem~\ref{thm:main}. 
The differential operator $D$ uses weak gradient
 and divergence operators $\nabla_M$ and $\div_M$
 acting on weighted $L^2$ spaces on $M$.
 For the definition and formulas for $\nabla$ (exterior derivative on scalar
 functions) and $\div$ (interior derivative on vector fields) acting on
 smooth sections, we refer to \cite[Sec.~11.2]{R2}.
 A vector field $f\in L^1_\loc(TM)$ is the weak gradient of a 
 function $u\in L^1_\loc(M)$ if 
$\scl{f}{\phi}= -\scl{u}{\div\phi}$ for all $\phi\in C^\infty_c(TM)$,
and we write $\nabla u=f$.
A function $f\in L^1_\loc(M)$ is the weak divergence of a
vector field $v\in L^1_\loc(TM)$ if
$\scl{f}{\phi}= -\scl{v}{\nabla\phi}$ for all $\phi\in C^\infty_c(M)$,
and we write $\div v=f$.

\begin{Lem}   \label{lem:pullpush}
Let $\rho: M\to N$ be a bi-Lipschitz map between 
Riemannian manifolds $M$ and $N$. Denote by 
$d\rho\in L^\infty(\End(TM,TN))$ and 
$J_\rho\in L^\infty(M)$ the Jacobian matrix and determinant of $\rho$.

(i) Let $u\in L^1_\loc(N)$ have weak gradient $\nabla u\in L^1_\loc(TN)$. Then the weak gradient of $u\circ\rho\in L^1_\loc(M)$ exists and equals
$$
  \nabla (u\circ \rho)= \rho^*(\nabla u),
$$
where $(\rho^* v)(x)= (d\rho_x)^* v(\rho(x))$, $x\in M$, denotes the pullback of a vector field $v$.

(ii)
Let $v\in L^1_\loc(TM)$ have weak divergence $\div v\in L^1_\loc(M)$.
Define the {\em Piola transform} $(\tilde \rho_* v)(\rho(x))= J_\rho(x)^{-1}(d\rho_x) v(x)$, $x\in M$, of the vector field $v$. Then the weak divergence of 
$\tilde \rho_* v\in L^1_\loc(N)$ exists and equals
$$
  \div (\tilde \rho_* v)= J_\rho^{-1}(\div v)\circ\rho^{-1}.
$$
\end{Lem}

The proof follows from a straightforward generalization of 
\cite[Lem.~10.2.4]{R2} to locally integrable differential forms
on manifolds.
The identities (i) and (ii) are dual in the sense that $-\div$ and 
$\tilde \rho_*$ are $L^2$ adjoints of $\nabla$ and $\rho^*$
respectively. The chain rule (i), and also (ii) via Hodge duality, are  
instances of the well known commutation between pullback
and exterior derivative.

\begin{Eg}  \label{ex:pullcoeffs}
Let $\rho: M\to N$ be a bi-Lipschitz map and consider
a divergence form equation $\div A\nabla u=0$ on $N$, 
where $A\in L^\infty(\End(TN))$.
Define $f= u\circ \rho$ on $M$.
Lemma~\ref{lem:pullpush}(i) shows that $\nabla u= (\rho^*)^{-1}\nabla f$
and Lemma~\ref{lem:pullpush}(ii) applied to $\rho^{-1}$ yields
$$
  0= J_\rho(\div A\nabla u)\circ \rho= \div(\tilde\rho_*^{-1}A (\rho^*)^{-1}) \nabla f.
$$
We refer to the coefficents 
$A_\rho= \tilde\rho_*^{-1}A (\rho^*)^{-1}= 
J_\rho (d\rho)^{-1}(A\circ\rho) (d\rho^*)^{-1}$
as the {\em pulled back coefficients} on $M$.
\end{Eg}

Let $\omega\in A^R_2(M)$ for some $R>0$.
We define $\nabla_M$ to be the closed operator
$\nabla_M: L^2(M,\omega)\to L^2(TM,\omega): u\mapsto \nabla u$ with domain
$$
  \Dom(\nabla_M)= W^{1,2}(M,\omega)=
  \sett{u\in L^2(M,\omega)}{\nabla u\in L^2(TM,\omega)}.
$$
More precisely, $\nabla u\in L^2(TM,\omega)$ means that
the weak gradient of $u$ exists and belongs to $L^2(TM,\omega)$.
We define $\div_M$ to be the closed operator
$\div_M: L^2(TM,\tfrac 1\omega)\to L^2(M,\tfrac 1\omega)
: v\mapsto \div v$ with domain
$$
  \Dom(\div_M)=
  \sett{v\in L^2(TM,\tfrac 1\omega)}{\div v\in L^2(M,\tfrac 1\omega)}.
$$
More precisely, $\div v\in L^2(M,\tfrac 1\omega)$ means that
the weak divergence of $v$ exists and belongs to $L^2(M,\tfrac 1\omega)$.
Using that multiplication $\omega: L^2(M,\omega)\to L^2(M,\tfrac 1\omega)$
and $\omega: L^2(TM,\omega)\to L^2(TM,\tfrac 1\omega)$ 
is an isometry, we obtain a closed operator
$$
    \div_{M,\omega}= \tfrac 1\omega\div_M \omega:
     L^2(TM,\omega) \rightarrow L^2(M,\omega),
$$ 
which is unitary equivalent to $\div_M$.

  Our definitions of $\dom(\nabla_M)$ and $\dom(\divv_M)$ above use the weak distributional definition of the differential operators on the smooth manifold and the natural weighted domains. The following lemma shows that this is equivalent to the more well-known definition involving closures from smooth functions.

\begin{Lem}    \label{lem:ddeladj}
  Let $M$ and $\omega$ be as in Theorem~\ref{thm:main}.
  Then the 
  smooth compactly supported functions $C^\infty_c(M)$ are dense in 
  $\Dom(\nabla_M)$, in graph norm,
  and $C^\infty_c(TM)$ is dense in $\Dom(\div_M)$, in graph norm.
  Moreover, $\nabla_M$ and $-\div_M$ are adjoint under the unweighted 
  $L^2(M)$ pairing in the sense of unbounded operators, or equivalently,
  $\nabla_M$ and $-\div_{M,\omega}$ are adjoint under the weighted
  $L^2(M,\omega)$ pairing in the sense of unbounded operators.
\end{Lem}

\begin{proof}
  From the coverings in Section~\ref{sec:covering}, we obtain
  a locally finite partition of unity $\sum_i\eta_i=1$ on $M$,
  with $0\le \eta_i\le 1$ smooth and supported on balls $B(x_i,T)$ with $0<T<R$
  and $\sup_{i,x} |\nabla \eta_i(x)|<\infty$.
  We first prove that $C^\infty_c(M)$ is dense in 
  $\Dom(\nabla_M)$.
  Given $u\in \Dom(\nabla_M)$ and $\epsilon>0$, we 
  can choose a finite index set $I$ such that
  $$
  \int_M\left( \Big|\sum_{i\notin I} \eta_i u\Big|^2+
  \Big|\nabla \sum_{i\notin I} \eta_i u\Big|^2 \right)\omega d\mu
  \lesssim
   \int_{\bigcup_{i\notin I} B(x_i,T)}\left( |u|^2+
  |\nabla u|^2 \right)\omega d\mu <\epsilon^2
  $$ 
  since $\int_M\left( |u|^2+
  |\nabla u|^2 \right)\omega d\mu<\infty$.
  Therefore it suffices to find, for each $i\in I$,  
  $\tilde u_i\in C^\infty_c(B(x_i,T))$ 
  such that 
\begin{equation}   \label{eq:tildeuitildevi}
   \int_{B(x_i,T)}\left( |\eta_i u-\tilde u_i|^2+
  |\nabla (\eta_i u-\tilde u_i)|^2 \right)\omega d\mu <(\epsilon/|I|)^2.
\end{equation}
  Indeed, $\sum_{i\in I} \tilde u_i\in C^\infty_c(M)$
  and $u-\sum_{i\in I} \tilde u_i= \sum_{i\notin I} \eta_i u+ 
  \sum_{i\in I} (\eta_i u-\tilde u_i)$ can be estimated by the 
  triangle inequality.
  To construct $\tilde u_i$, we use a chart 
  $\varphi: B(x_i,T)\to \Omega\subset \R^n$, and
  consider $v_i=(\eta_i u)\circ\varphi^{-1}$.
  It follows from Lemma~\ref{lem:pullpush}(i), that 
  $v_i\in L^2(\Omega, \tilde\omega)$, $\nabla v_i\in L^2(T\Omega,\tilde\omega)$, where $\tilde\omega= \omega\circ\varphi^{-1}$.
  Note that the weighted estimates holds since pullback  by
  $\varphi$ and $\varphi^{-1}$ are bounded pointwise, and hence 
  bounded between the 
  weighted $L^2$ spaces.
  Next we use \cite[Lem.~2.2]{ARR} to construct $\tilde v_i\in C^\infty_c(\Omega)$ such that
 $$
   \int_{\Omega}\left( |v_i-\tilde v_i|^2+
  |\nabla (v_i-\tilde v_i)|^2 \right)\tilde\omega dx <(\epsilon/|I|)^2.
  $$ 
  Note that $\tilde\omega$ being defined only in $\Omega$ does not 
  pose a problem since we use a mollifier $\varphi_\epsilon\in C^\infty_c(\R^n)$
  with support near $x=0$ in \cite[Lem.~2.2]{ARR}.
  (Alternatively, an auxiliary extended $A_2(\R^n)$ weight can be constructed as in Lemma~\ref{lem:extweight} below.)
  Setting $\tilde u_i= \tilde v_i\circ\varphi$ yields 
  \eqref{eq:tildeuitildevi} in $B(x_i,T)$.
  
  A similar argument shows that $C^\infty_c(TM)$ is dense in $\Dom(\div_M)$. We replace $\nabla$ by $\div$ in the argument above,  
  set $v_i=\tilde \varphi_*(\eta_i u)$ and use Lemma~\ref{lem:pullpush}(ii).
  In mollifying $v_i$ using \cite[Lem.~2.2]{ARR}, we note that convolutions
  commute also with the divergence. 
  
  The duality
  $$
  \int_M \scl{\nabla_M u}v d\mu= -\int_M\scl u{\div_{M} v} d\mu
  $$  
  holds by construction of $\nabla_M$ for all $u\in\Dom(\nabla_M)$
  and all $v\in C^\infty_c(TM)$ and hence, by the above density result,
  for all $v\in \Dom(\div_M)$.
  That $\nabla_M$ and $-\div_M$ are adjoint in the sense of 
  unbounded operators is then immediate from the definition
  of these operators.
\end{proof}

  Henceforth, we fix the vector bundle
\begin{equation}\label{eq:Vdefn}
\mathcal{V}=(M\times \C) \oplus (M\times \C) \oplus TM,
\end{equation}
 
with the induced bundle metric, so that $L^2(\mathcal{V},\omega) = L^2(M,\omega) \oplus L^2(M,\omega) \oplus L^2(TM,\omega)$. 
For a weight $\omega\in A^R_2(M)$, $R>0$, define the self-adjoint operator
\begin{equation}  \label{eq:Ddefn}
D=\begin{bmatrix} 0& I & -\div_{M,\omega}\\ I & 0 & 0 \\ \grad_M & 0 & 0 \end{bmatrix}:
\Dom(D) \subseteq L^2(\mathcal{V},\omega)\rightarrow L^2(\mathcal{V},\omega),
\end{equation}
with $\Dom(D)= W^{1,2}(M,\omega) \oplus L^2(M,\omega) \oplus \Dom(\div_{M,\omega})$. 
Note that $D$ has closed range $\Ran(D)$, due to the zero order identity terms in \eqref{eq:Ddefn}.

Consider $B\in L^\infty(\End(\mathcal{V}))$ and define a bounded multiplication operator $B:L^2(\mathcal{V},\omega) \rightarrow L^2(\mathcal{V},\omega)$ by setting
\begin{equation}  \label{eq:Bsplit}
(Bu)(x) = (B(x))(u(x)), \qquad  x\in M,  u\in L^2(\mathcal{V},\omega).
\end{equation}
In the splitting $\mathcal{V}_x=\C \oplus \C \oplus T_xM$, we write
\begin{equation*}  
u(x)= \begin{bmatrix} u_\no(v) \\ u_0(x) \\ u_\ta(x) \end{bmatrix}
\quad\text{and}\quad
B(x) = \begin{bmatrix} B_{\no\no}(x) & B_{\no 0}(x) & B_{\no\ta}(x)\\ B_{0\no}(x) & B_{00}(x) & B_{0\ta}(x) \\ B_{\ta\no}(x) & B_{\ta 0}(x) & B_{\ta\ta}(x) \end{bmatrix}, \qquad  x\in M,
\end{equation*}
where $B_{\no\no}, B_{\no 0}, B_{0\no}, B_{00} \in L^\infty(M)$, $B_{\ta\no}, B_{\ta 0} \in L^\infty(TM)$, $B_{\no\ta}, B_{0\ta}\in L^\infty(T^*M)=L^\infty(TM)$ and $B_{\ta\ta} \in L^\infty(\End(TM))$.
We assume that $B$ is accretive on $\Ran(D)$ in the sense that
\begin{equation}\label{eq:acc}
\kappa_B = \inf \{\re \langle Bu,u\rangle_{L^2(\mathcal{V},\omega)} / \|u\|_{L^2(\mathcal{V},\omega)}^2 : u\in \Ran(D), u\neq 0\} > 0.
\end{equation}

\subsection{Functional calculus}   \label{sec:funccal}

We recall that a closed and densely defined operator $T$ in a Hilbert
space $\mH$ is called bisectorial, if its spectrum is contained in a closed bisector
$$
  S_\nu=\sett{z\in \C}{|\arg(z)|\le\nu}\cup \{0\}\cup \sett{z\in\C}{|\arg(-z)|\le\nu},
$$
for some $0\le \nu<\pi/2$, with resolvent bounds 
$\|(\lambda I-T)^{-1}\|\lesssim 1/\dist(\lambda, S_\nu)$.
It follows from \cite[Prop.~3.3]{AAMc2} that the operators $BD$ and $DB$ from Section~\ref{sec:weightsops} are bisectorial in 
$L^2(\mathcal{V},\omega)$ for any coefficients $B$ satisfying \eqref{eq:acc}.   The proof therein, which treats the case when $\omega \equiv 1$ and $\mathcal{V}$ is a trivial bundle over $\R^n$, extends immediately to the setting considered here.  
Moreover we have a topological splitting
\begin{equation}   \label{eq:dbsplit}
   L^2(\mathcal{V},\omega)= \Ran(DB)\oplus \nul(DB),
\end{equation}
where in general the two closed subspaces are not orthogonal. 
Injective bisectorial operators $T$ have a bounded
$H^\infty(S^o_\theta)$ functional calculus
for any $\nu<\theta<\pi/2$, provided quadratic estimates
\begin{equation}   \label{eq:qeequiv}
   \int_0^\infty \|\psi_t(T)u\|^2 \frac{dt}{t} \eqsim \|u\|^2, \qquad u\in \mH,
\end{equation}
hold.
Here $\psi(\lambda)= \lambda/(1+\lambda^2)$ and $\psi_t(\lambda)=
\psi(t\lambda)$, and
boundedness of the $H^\infty(S^o_\theta)$ functional calculus means 
that we have an estimate
$$
   \|b(T)\|_{\mH\to\mH}\lesssim \sup_{\lambda\in S^o_\theta}|b(\lambda)|
$$
for all bounded holomorphic functions $b:S^o_\theta\to \C$, where
$$
  S^o_\theta=\sett{z}{|\arg(z)|<\theta}\cup \sett{z}{|\arg(-z)|<\theta}.
$$
We refer to \cite[Sec.~6.1]{AA1} for a short proof of these well known facts,
and we refer to \cite{AlbrechtDuongMcIntosh, AuscherNM:97} for more background,
where the generalization from sectorial to bisectorial operators is
straightforward.
In proving quadratic estimates \eqref{eq:qeequiv}, it suffices to prove
the estimate $\lesssim$ for $T$ and its adjoint $T^*$.
Indeed, by a duality argument the latter implies the estimate
$\gtrsim$. See \cite{AlbrechtDuongMcIntosh}.

For non-injective bisectorial operators $T$, the above applies to
the restriction of $T$ to the invariant subspace $\clos{\Ran(T)}$,
since this gives an injective bisectorial operator,
and the splitting $\mH=\clos{\ran(T)}\oplus \nul(T)$ allows an 
extension of the functional calculus, setting
$$
  b(T)= b(T|_{\clos{\ran(T)}})\oplus b(0)I_{\nul(T)}
$$
  for $b: S_\theta^o\cup\{0\}\to \C$ with 
$b|_{S_\theta^o}\in H^\infty(S_\theta^o)$, where $T|_{\clos{\ran(T)}}=T\mathbb{P}$, $I_{\nul(T)}=I-\mathbb{P}$ and $\mathbb{P}$ denotes the projection onto $\clos{\ran(T)}$ along $\nul(T)$.  
We refer to the function $b$ defining the operator $b(T)$ as the {\em symbol}
of $b(T)$.

The following abstract argument allows us to 
perturb the quadratic estimates from Theorem~\ref{thm:main}.

\begin{Lem}  \label{lem:bddpert}
  Let $T$ be a bisectorial operator in a Hilbert space $\mH$, with local quadratic estimates
  $$
    \int_0^1 \|\psi_t(T) u\|^2 \frac{dt}t\lesssim \|u\|^2,\qquad u\in \mH.
  $$
  Let $V:\mH\to \mH$ be a bounded operator such that $T+V$, with domain $\Dom(T)$, 
  is a bisectorial operator with closed range.
  Then
$$
\int_0^\infty \|\psi_t(T+V)u\|^2 \frac{dt}{t} \lesssim \|u\|^2, \qquad u \in \mH.
$$
\end{Lem}

\begin{proof}
  We have the estimate 
$$
  \|(I+it (T+V))^{-1}- (I+itT)^{-1} \| 
  = \| (I+it (T+V))^{-1}(tV)(I+itT)^{-1}  \|\lesssim |t|
$$
between resolvents, for $t\in \R$.
Subtraction by the corresponding estimate with $t$ replaced by $-t$,
shows
$$
\int_0^1 \|\psi_t(T+V)u\|^2 \frac{dt}{t} \lesssim \int_0^1 (\|\psi_t(T)u\|^2 + t^2\|u\|^2)\frac{dt}{t}\lesssim \|u\|^2,\qquad u\in \mH.
$$
Next we note that 
$\psi_t(T+V)u = t(T+V)\big((I+t^2(T+V)^2)^{-1}u\big) \in \Ran(T+V)$, which by bisectoriality is
a subspace complementary to $\Nul(T+V)$. Assuming that $\Ran(T+V)$ is closed, it follows from
the Open Mapping Theorem that
$$
  \|\psi_t(T+V)u\| 
 \lesssim
 t^{-1}\|t (T+V)\psi_t(T+V)u\|
\lesssim t^{-1} \|u\|,\qquad u\in \mH.
$$
This proves the estimate $\int_1^\infty \|\psi_t(T+V)u\|^2 \frac{dt}{t} \lesssim \|u\|^2$ for $u\in\mH$.
\end{proof}

\section{Proof of quadratic estimates}\label{sect:loc}

We prove Theorem~\ref{thm:main} in this section.
The proof is split into subsections and lemmas to increase readability.
For the whole proof, we fix 
$C_{gbg}=4$ and $\beta=1/2$ and let $r_H>0$ be the harmonic radius
of $M$, for these constants, as in Section~\ref{sec:covering}.   These choices for $C_{gbg}$ and $\beta$ are allowed, and determine the value of $r_H$, as per the discussion preceding \eqref{eq:hr0}. Moreover, since we will not require the maximal harmonic radius $r_H$, we can and will assume without loss of generality that $r_H<R$.  
Fix $0<T<r_H/64$ and denote balls covering $M$, as in Section~\ref{sec:covering}, by 
$$
  B_i= B(x_i,T)\subset  M,
$$
and write $B_i^*= B(x_i,r_H)\subset  M$ for the enlarged balls on which
we have charts with the estimate \eqref{eq:hr0}.
We write $\Omega\subset \R^n$ and $\Omega^*\subset\R^n$ for Euclidean balls centered at $0$
with radii  $16T$ and $r_H/2$ respectively.
Since $C_{gbg}=4$, length dilation is at most $2$ and we have $\varphi_i(4B_i)\subset \tfrac 12\Omega$
and $\varphi_i^{-1}(\Omega^*)\subset B_i^*$ for the
harmonic coordinate charts $\varphi_i$, where we
without loss of generality assume that $\varphi_i(x_i)=0$.

Besides the standing assumption that $T<r_H/64$, 
ensuring that $2\Omega\subset \Omega^*$, we
require that $T>0$ satisfies 
\eqref{eq:TreqCB} and \eqref{eq:Treq2} below.

\subsection{Localization on $M$} 

Write $\psi_t(\lambda)= t\lambda/(1+t^2\lambda^2)$. 
As in Lemma~\ref{lem:bddpert} it suffices to prove that 
\begin{equation}\label{eq:mainred}
\int_0^T \|\psi_t(DB) u\|_{L^2(\mathcal{V},\omega)}^2 \frac{dt}{t} \lesssim \|u\|_{L^2(\mathcal{V},\omega)}^2,
\end{equation}
since $\ran(D)$ is closed.
Moreover we may assume $u\in\ran(D)$, using the splitting \eqref{eq:dbsplit}.
We write
\begin{align*}
&\int_0^T \|\psi_t(DB)u\|_{L^2(\mathcal{V},\omega)}^2 \frac{dt}{t} \\
&\lesssim  \int_0^T \sum_{i\in\N} \left(\|\ca_{B_i} \psi_t(DB) \ca_{4B_i} u\|_{L^2(\mathcal{V},\omega)}^2 + \|\ca_{B_i} \psi_t(DB) \ca_{M\setminus 4B_i} u\|_{L^2(\mathcal{V},\omega)}^2\right) \frac{dt}{t}\\
&\leq \sum_{i\in\N} \int_0^T \|\ca_{4B_i} \psi_t(DB) \ca_{4B_i} u\|_{L^2(\mathcal{V},\omega)}^2 \frac{dt}{t} + \int_0^T \sum_{i\in\N} \|\ca_{B_i} \psi_t(DB) \ca_{M\setminus 4B_i} u\|_{L^2(\mathcal{V},\omega)}^2 \frac{dt}{t} \\
&= I_1 + I_2.
\end{align*}
 The main term $I_1$ is estimated in the following sections.
 In the present section we estimate the error term $I_2$, using the off-diagonal decay in Lemma~\ref{lem:ODest} to offset volume growth on the manifold.
We write 
$$
\langle \alpha \rangle = \max \{\alpha, 1\},\qquad \alpha>0.
$$ 

\begin{Lem}\label{lem:ODest} 
There exists a constant $C_B>0$, depending only on $\kappa_B$ and $\|B\|_\infty$, such that for each $N\geq0$ it holds that
\[
\|\ca_E \psi_t(DB)\ca_F\|_{\mathcal{L}(L^2(\mathcal{V},\omega))} \lesssim \left\langle\frac{d(E,F)}t\right\rangle^{-N} \exp\left(-C_B \frac{d(E,F)}{t}\right),
\]
for all $t>0$ and all measurable sets $E,F\subseteq M$, where the implicit constant depends on $\kappa_B$, $\|B\|_\infty$ and $N$.
\end{Lem}

\begin{proof}
This follows as in \cite[Prop.~5.2]{MM}, which in turn combines \cite[Lem.~5.3]{CMcM} with \cite[Prop.~5.2]{AKMc}, since $DB$ is a bisectorial operator in $L^2(\mathcal{V},\omega)$ and the commutator $[\eta,DB]=\eta DB - DB\eta$ satisfies the pointwise identity
\[
|[\eta,DB]u(x)|_{\mathcal{V}_x} = |[\eta,D] (Bu) (x)|_{\mathcal{V}_x}  \leq |\grad_M\eta(x)|_{T_xM} \|B\|_\infty |u(x)|_{\mathcal{V}_x},
\]
for all $\eta\in C^\infty_c(M)$, $u\in\Dom(DB)$ and $x\in M$.
\end{proof}

The following lemma can be deduced from the proof of \cite[Lem.~4.4]{MM}, but the short proof is included for convenience.

\begin{Lem}\label{lem:cubesup}
Let $\lambda= \sqrt{(n-1)K}$.
If $N>n$ and $m>\lambda T$, then
\begin{equation}\label{eq:cubes0}
\sup_{i\in\N} \sum_{j\in\N}  \frac{\mu(B_j)}{\mu(B_i)} \left\langle\frac{t}{d(B_i,B_j)}\right\rangle^{N} e^{-m {d(B_i,B_j)}/{t}} \lesssim 1, \qquad  t\in(0,T],
\end{equation}
where the implicit constant depends on $n, \lambda, N,m$ and $T$. 
If in addition $\alpha >0$, then
\begin{equation}\label{eq:cubes1}
\sup_{i\in\N} \sum_{\substack{j\in\N:\\ d(B_i,B_j)\geq\alpha}}  \frac{\mu(B_j)}{\mu(B_i)} \left\langle\frac{t}{d(B_i,B_j)}\right\rangle^{N} e^{-m {d(B_i,B_j)}/{t}} \lesssim t^N, \qquad  t\in(0,T],
\end{equation}
where the implicit constant depends on $n, \lambda, N,m,T$ and $\alpha$.
\end{Lem}

\begin{proof}
Let $T>0$, $N>n$, $m>\lambda T$ and $\alpha>0$. Set $\sigma=m/\lambda T >1$ and for each $i\in\N$, isolate the balls intersecting concentric annuli around $B_i$ by defining
\[
\mathcal{A}_k(B_i) =
\begin{cases}
\{B_j : d(B_i,B_j) < \alpha\},  &{\rm if}\quad k=0, \\
\{B_j : \alpha\sigma^{k-1}\leq d(B_i,B_j) < \alpha\sigma^k\},  &{\rm if}\quad k\in\N.
\end{cases}
\]
A straightforward calculation shows that
\[
{\textstyle\bigcup_{B_j\in\mathcal{A}_k(B_i)}} B_j \subseteq B(x_i,3T+\alpha\sigma^k), \qquad  k\in\N_0,\  i\in\N.
\]
It follows from~\eqref{eq:expgrowth} and the 
disjointness of the balls $((1/2)B_i)_{i\in\N}$ that
\[
\sum_{B_j\in\mathcal{A}_k(B_i)} \mu(B_j)
\lesssim e^{\lambda T} \mu\left({\textstyle\bigcup_{B_j\in\mathcal{A}_k(B_i)}} (1/2) B_j \right) \lesssim \sigma^{kn}e^{\lambda \alpha\sigma^k} \mu(B_i), \qquad  k\in\N_0, i\in\N.
\]
For all $t\in(0,T]$, since $m- \sigma \lambda t \geq m - \sigma \lambda T=0$, we then have
\begin{align*}
\sup_{i\in\N} &\sum_{j\in\N}  \frac{\mu(B_j)}{\mu(B_i)} \left\langle\frac{t}{d(B_i,B_j)}\right\rangle^{N} e^{-m {d(B_i,B_j)}/{t}}
\\
&=\sup_{i\in\N} \sum_{k=0}^\infty \sum_{B_j\in\mathcal{A}_k(B_i)}  \frac{\mu(B_j)}{\mu(B_i)} \left\langle\frac{t}{d(B_i,B_j)}\right\rangle^{N} e^{-m {d(B_i,B_j)}/{t}} \\
&\leq \sup_{i\in\N} \frac{1}{\mu(B_i)} \left(\sum_{B_j\in\mathcal{A}_0(B_i)} \mu(B_j) +
\sum_{k=1}^\infty (\alpha\sigma^{k-1}/t)^{-N}e^{-m\alpha\sigma^{k-1}/t} \sum_{B_j\in\mathcal{A}_k(B_i)} \mu(B_j) \right)
\\
&\lesssim 1 + t^N \sum_{k=1}^\infty \sigma^{-k(N-n)}e^{-(m-\sigma\lambda t)\alpha\sigma^{k-1}/t}
\\
&\lesssim 1 + t^N,
\end{align*}
which proves~\eqref{eq:cubes0}, and we obtain~\eqref{eq:cubes1} by excluding the term $k=0$ in the sum.
\end{proof}

To estimate $I_2$,   on recalling that $\lambda= \sqrt{(n-1)K}$,   we require that
\begin{equation} \label{eq:TreqCB}
  T< C_B/\lambda,
\end{equation}
where $C_B>0$ is the constant from Lemma~\ref{lem:ODest}. 
  In particular, choosing $N>n$ and setting $m=C_B$, we can then combine the off-diagonal estimate in Lemma~\ref{lem:ODest} with the geometric control afforded by Lemma~\ref{lem:cubesup} to obtain  
\begin{align*}
&\sum_{i\in\N} \|\ca_{B_i} \psi_t(DB) \ca_{M\setminus 4B_i} u\|_{L^2(\mathcal{V},\omega)}^2 \\
&\leq \sum_{i\in\N} \Bigg(\sum_{\substack{j\in \N:\\d(B_i,B_j)\geq T}} \Bigg(\frac {\mu(B_j)}{\mu(B_i)}\frac{\mu(B_i)}{\mu(B_j)}\Bigg)^{1/2} \|\ca_{B_i} \psi_t(DB) \ca_{B_j} u\|_{L^2(\mathcal{V},\omega)}\Bigg)^2\\
&\leq \sum_{i\in\N} \Bigg(\sum_{\substack{j\in \N:\\d(B_i,B_j)\geq T}} \frac{\mu(B_j)}{\mu(B_i)}\|\ca_{B_i} \psi_t(DB) \ca_{B_j}\|\Bigg) \\
&\qquad\quad \times
\Bigg(\sum_{j\in\N} \frac{\mu(B_i)}{\mu(B_j)} \|\ca_{B_i} \psi_t(DB) \ca_{B_j}\| \|\ca_{B_j} u\|_{L^2(\mathcal{V},\omega)}^2\Bigg)\\
&\lesssim \Bigg(\sup_{i\in\N} \sum_{\substack{j\in \N:\\d(B_i,B_j)\geq T}} \frac{\mu(B_j)}{\mu(B_i)} \left\langle\frac{t}{d(B_i,B_j)}\right\rangle^{N} e^{-C_B {d(B_i,B_j)}/{t}} \Bigg)\\
&\qquad\quad \times
\Bigg(\sup_{j\in\N} \sum_{i\in\N} \frac{\mu(B_i)}{\mu(B_j)} \left\langle\frac{t}{d(B_i,B_j)}\right\rangle^{N} e^{-C_B {d(B_i,B_j)}/{t}}\Bigg) \Bigg(\sum_{j\in\N} \|\ca_{B_j} u\|_{L^2(\mathcal{V},\omega)}^2\Bigg)\\
&\lesssim t^N \|u\|_{L^2(\mathcal{V},\omega)}^2
\end{align*}
for all $t\in(0,T]$. This proves that $I_2 \lesssim \|u\|_{L^2(\mathcal{V},\omega)}^2$ for all $u \in L^2(\mathcal{V},\omega)$.

\subsection{Pullback from $M$ to $\R^n$}\label{sect:LQE}

To prove \eqref{eq:mainred}, and thus Theorem~\ref{thm:main}, it remains to prove that 
\begin{equation}\label{eq:mainred2}
I_1 = \sum_{i\in\N} \int_0^T \|\ca_{4B_i} \psi_t(DB) \ca_{4B_i} u\|_{L^2(\mathcal{V},\omega)}^2 \frac{dt}{t}
\lesssim \|u\|_{L^2(\mathcal{V},\omega)}^2,
\qquad u \in L^2(\mathcal{V},\omega).
\end{equation}
  We estimate \eqref{eq:mainred2} term-wise, fix $i$ and
use the
harmonic coordinate chart $\varphi_i$ for the vector bundle $\mathcal{V}$ from \eqref{eq:Vdefn}. 
Suppressing $i$ in notation, we write $\rho= \varphi_i^{-1}$.
Define the pulled back weight $\omega_\rho= \omega \circ \rho$ in 
$\Omega^*$.
We use $\rho$ to locally intertwine the operator $DB$
from Section~\ref{sec:weightsops}
in $L^2(\rho(2\Omega), \mV, \omega)$
and an operator $D_\rho B_\rho$
in $L^2(2\Omega, \C^{2+n}, \omega_\rho)$,
using Lemma~\ref{lem:pullpush}.
Define a pushforward operator $T_\rho: L^2(2\Omega,\C^{2+n},\omega_\rho) \rightarrow L^2(\rho(2\Omega),\mathcal{V},\omega)$ by
\[
 T_\rho \begin{bmatrix} u_\no \\ u_0 \\ u_\ta \end{bmatrix}(\rho(x))
= \begin{bmatrix} u_\no(x) \\  u_0(x)/J_\rho(x) \\ (d\rho)_{x} u_\ta(x)/J_\rho(x)\end{bmatrix}, \qquad x\in 2\Omega,
\]
for $u_\no,u_0\in L^2(2\Omega,\C,\omega_\rho)$ and $u_\ta\in L^2(2\Omega,\C^{n},\omega_\rho)$. 
The weighted $L^2$ adjoint of $T_\rho$ is 
$T^*_\rho:L^2(\rho(2\Omega),\mathcal{V},\omega) \rightarrow L^2(2\Omega,\C^{n+1},\omega_\rho)$ given by
\[
  T^*_\rho \begin{bmatrix} v_\no \\ v_0 \\ v_\ta \end{bmatrix}(x)
= \begin{bmatrix} J_\rho(x) v_\no(\rho(x)) \\ v_0(\rho(x)) \\ (d\rho_x)^* v_\ta(\rho(x)) \end{bmatrix}, \qquad  x\in 2\Omega,
\]
for $v_\no,v_0\in L^2(\rho(2\Omega),\C,\omega)$ and $v_\ta\in L^2(\rho(2\Omega),TM,\omega)$.

The Euclidean version of differential operator \eqref{eq:Ddefn}
is 
\[
D_\rho=\begin{bmatrix} 0& I& -\div_{\omega_\rho} \\ I & 0 & 0 \\ \grad & 0 & 0\end{bmatrix},
\]
where $\div_{\omega_\rho}= \omega_\rho^{-1} \div\omega_\rho$ 
(for now acting on functions defined on $2\Omega$)
and $\nabla$ and $\div$
are the weak gradient and divergence in $\R^n$.
In Section~\ref{sec:extension}, we extend $\omega_\rho$ and
define $D_\rho$ as a self-adjoint operator on $\R^n$.

\begin{Lem}   \label{lem:locintertw}
Let $B$ be a bounded and accretive multiplication operator 
on $L^2(M,\omega)$ as in Section~\ref{sec:weightsops}.
Define the bounded multiplication operator $B_\rho=T_\rho^{-1} B{T_\rho^*}^{-1}$
on $L^2(2\Omega,\C^{2+n},\omega_\rho)$.
Then
\begin{equation}\label{eq:DBMDBR}
T_\rho^* DB  v = D_\rho B_\rho T_\rho^* v\qquad \text{on } 2\Omega,
\end{equation}
for all $v \in \Dom(DB)$.
\end{Lem}

\begin{proof}
  We note that $T_\rho$ is an invertible multiplication operator.
  In particular $B_\rho$ is a bounded multiplication operator.
  To prove the local intertwining formula \eqref{eq:DBMDBR} , it 
  suffices to verify that $T^*_\rho D= D_\rho T_\rho^{-1}$ on $\rho(2\Omega)$.
  Acting on $\begin{bmatrix}v_\no \\ v_0 \\ v_\ta\end{bmatrix}$, this amounts to 
  $$
     \begin{bmatrix} J_\rho (v_0\circ \rho)-J_\rho(\div_{M,\omega} v_\ta)\circ \rho \\
      v_\no\circ \rho \\ d\rho^*(\nabla_M v_\no)\circ \rho \end{bmatrix}
      =
           \begin{bmatrix} J_\rho (v_0\circ \rho)-
            \div_{\rho,\omega_\rho}J_\rho d\rho^{-1} (v_\ta\circ \rho) \\
      v_\no\circ \rho \\ \nabla (v_\no\circ\rho) \end{bmatrix}.
  $$
  We conclude using Lemma~\ref{lem:pullpush}.
\end{proof}

\subsection{Extension of weights and coefficients} \label{sec:extension}

In this section, we extend the weight $\omega_\rho$ and 
the coefficients $B_\rho$ to $\R^n$, and add a lower order term $r\mathbb{P}_0$ to $B_\rho$,
to obtain a bisectorial operator $D_\rho B_{\rho,r}$ on $\R^n$ for 
which weighted quadratic estimates are proved by tweaking the 
proof from \cite{ARR} in Section~\ref{sec:finalbits}.
For the following extension argument for the weight $\omega_\rho$,
we require the more precise upper bound
\begin{equation}  \label{eq:Treq2}
  T<r_H/(64(1+2\sqrt n)),
\end{equation} 
  which will also suffice for our ultimate purpose of obtaining \eqref{eq:mainred} for some $T>0$.  

\begin{Lem}   \label{lem:extweight}
There exists a weight $\widetilde \omega_\rho\in A_2(\R^n)$ such that
$$
\widetilde \omega_\rho= \omega_\rho\qquad\text{on } 2\Omega,
$$
with $[\widetilde\omega_\rho]_{A_2(\R^n)}$ depending only
$[\omega]_{A_2^R(M)}$, $R$, $n$, $r_E$ and $K$.
\end{Lem}

\begin{proof} 
We verify condition (ii) in \cite[Thm.~IV.5.6]{GR} for $\omega_\rho$ on
$2\Omega$, which guarantees that if $\omega_\rho^{1+\epsilon}\in A_2(2\Omega)$ for some $\epsilon>0$, then there exists $\widetilde \omega_\rho\in A_2(\R^n)$ as required by the lemma .
 It suffices to prove
$$
  \sup_{Q}\left(\frac 1{|Q|}\int_{Q\cap 2\Omega} \omega_\rho^{1+\epsilon} dx \right)
  \left( \frac 1{|Q|}\int_{Q\cap 2\Omega} \omega_\rho^{-1-\epsilon} dx\right)<\infty,
$$
with supremum taken over all cubes $Q\subset \R^n$,
for some $\epsilon>0$.
To this end, we first note that it suffices to consider $Q$ with
sidelength $\ell(Q)\le 2r$ and intersecting $2\Omega$, 
where $r= 32T$ is the radius of $2\Omega$.
(If $\ell(Q)>2r$, we may shrink $Q$ by moving the corner of $Q$ 
farthest away from $0$, without changing $Q\cap 2\Omega$.)
For such $Q$, we estimate by increasing the domains of integration
from $Q\cap 2\Omega$ to $Q$, since 
$Q\subset \Omega^*$ by \eqref{eq:Treq2}.
Having done this,
by reverse H\"older estimate and $A_p$
bootstrapping, see \cite[Thm.~9.2.2, 9.2.5]{G}, it suffices to 
prove 
the estimate for $\epsilon=0$.
We need to show a uniform bound on 
$$
\left(\frac 1{|Q|}\int_{Q} \omega_\rho dx \right)
  \left( \frac 1{|Q|}\int_{Q} \omega_\rho^{-1} dx\right).
$$ 
Changing variables to $B_i^*\subset M$, 
since $C_{gbg}=4$ and we assume $r_H<R$, the bound now
follows from $[\omega]_{A_2^R(M)}<\infty$.
\end{proof}

  Fix an $A_2(\R^n)$ weight $\widetilde \omega_\rho$ as in Lemma~\ref{lem:extweight}. This extension of $\omega_\rho$ from $\Omega$ to~ $\R^n$ can be constructed as in \cite[Thm.~IV.5.6]{GR} by factorising $\omega_\rho$ and applying maximal functions to the zero extensions of its components.  
Define the self-adjoint operator
\[
D_\rho=\begin{bmatrix} 0& I& -\div_{\widetilde\omega_\rho} \\ I & 0 & 0 \\ \grad & 0 & 0\end{bmatrix}
\]
in $L^2(\mathbb{R}^n,\C^{2+n},\widetilde\omega_\rho)$.
We next turn to the extension of $B_{\rho}$ to a bounded multiplication 
operator that is accretive on the range of $D_\rho$.
(A related construction is in \cite[App.~A]{AuscherBortzEgertSaari:19}.)
Define
\begin{equation}\label{eq:BrhoR}
B_{\rho,r} = B_\rho (\eta_0^2 I) + (1-\eta_0^2)I + r\mathbb{P}_0,
\end{equation}
where 
$$
\mathbb{P}_0 = \begin{bmatrix}
0 & 0 & 0 \\ 0 & I & 0 \\ 0 & 0 & 0 \end{bmatrix}.
$$
  The parameter $r$ will be fixed below to ensure accretivity of $B_{\rho,r}$ on $\ran(D_\rho)$, which is crucial, whilst $\eta_0:\R^n\rightarrow [0,1]$
denotes a fixed smooth cut-off function such that  
$$
 \sppt \eta_0 \subset 2\Omega \quad \text{ and } \quad \eta_0\equiv1 \text{ on } 
\Omega. 
$$
Recall that $\widetilde\omega_\rho=\omega_\rho$ on $2\Omega$.    The following lemma provides for our choice of $r$.

\begin{Lem}
There exists $r< \infty$, depending only on $n$, $\kappa_B$ and $\|B\|_\infty$, such that $B_{\rho,r}$ is accretive on $\Ran(D_\rho)$ in the sense of \eqref{eq:acc} with
\[
\kappa_{B_{\rho,r}} \geq \min\{\tfrac{1}{4} c^{-1}\kappa_B,1\}
> 0,
\]
where $c<\infty$ is a positive constant depending only on $n$.
\end{Lem} 

\begin{proof}
Let $u\in \Ran(D_\rho)$, so $u_\ta=\nabla u_0$, where $u_\no \in L^{2}(\mathbb{R}^n,\widetilde\omega_\rho)$ and $u_0\in W^{1,2}(\mathbb{R}^n,\widetilde\omega_\rho)$. Now we use $\nabla (\eta_0 u_{0}) = \eta_0\nabla u_{0} + (\nabla \eta_0) u_{0}$ to write
\begin{align*}
\langle B_\rho (\eta_0^2 u), u  \rangle 
&= \langle B_\rho 
\begin{bmatrix} \eta_0 u_\no \\ \eta_0 u_{0} \\ \eta_0\nabla u_{0} \end{bmatrix},
\begin{bmatrix} \eta_0 u_\no \\ \eta_0 u_{0} \\ \eta_0\nabla u_{0} \end{bmatrix}  \rangle \\
&= \langle B_\rho 
\begin{bmatrix} \eta_0 u_\no \\ \eta_0 u_{0} \\ \nabla (\eta_0 u_{0}) \end{bmatrix},
\begin{bmatrix} \eta_0 u_\no \\ \eta_0 u_{0} \\ \nabla (\eta_0 u_{0}) \end{bmatrix}  \rangle - E,
\end{align*}
where the error term is
\begin{align*}
E = \langle B_\rho 
\begin{bmatrix} 0 \\ 0 \\ (\nabla \eta_0) u_{0} \end{bmatrix}, \eta_0 u \rangle
 &+\langle B_\rho 
\eta_0 u,
\begin{bmatrix} 0 \\ 0 \\ (\nabla \eta_0) u_{0}\end{bmatrix} 
\rangle
+\langle B_\rho 
\begin{bmatrix} 0 \\ 0 \\ (\nabla \eta_0) u_{0}  \end{bmatrix},
\begin{bmatrix} 0 \\ 0 \\ (\nabla \eta_0) u_{0}  \end{bmatrix}
\rangle
\end{align*}
and we use the $L^2(\R^n,\C^{2+n},\widetilde\omega_\rho)$ inner-product.
Lemma~\ref{lem:pullpush} shows that
$$
{T_\rho^*}^{-1}\begin{bmatrix} \eta_0 u_\no \\ \eta_0 u_{0} \\ \nabla (\eta_0 u_{0}) \end{bmatrix} = \begin{bmatrix}J_\rho^{-1}(\eta_0 u_\no)\circ\rho^{-1} \\ (\eta_0 u_{0})\circ\rho^{-1} \\ \nabla_M ((\eta_0 u_{0})\circ\rho^{-1}) \end{bmatrix}
$$
belongs to $\Ran(D)$. The accretivity \eqref{eq:acc} of $B$ on $\Ran(D)$  then implies that  
\begin{align*}
&\re \langle B_\rho 
\begin{bmatrix} \eta_0 u_\no \\ \eta_0 u_{0} \\ \nabla (\eta_0 u_{0}) \end{bmatrix},
\begin{bmatrix} \eta_0 u_\no \\ \eta_0 u_{0} \\ \nabla (\eta_0 u_{0}) \end{bmatrix}  \rangle_{L^2(\Omega,\C^{2+n},\widetilde\omega_\rho)} \\
&= \re \langle B{T_\rho^*}^{-1}
\begin{bmatrix} \eta_0 u_\no \\ \eta_0 u_{0} \\ \nabla (\eta_0 u_{0}) \end{bmatrix},{T_\rho^*}^{-1}
\begin{bmatrix} \eta_0 u_\no \\ \eta_0 u_{0} \\ \nabla (\eta_0 u_{0}) \end{bmatrix}\rangle_{L^2(\mathcal{V},\omega)} \\
&\geq \kappa_B \|{T_\rho^*}^{-1}
\begin{bmatrix} \eta_0 u_\no \\ \eta_0 u_{0} \\ \nabla (\eta_0 u_{0}) \end{bmatrix}
\|_{L^2(\mathcal{V},\omega)}^2 
\geq c^{-1} \kappa_B
\|\begin{bmatrix} \eta_0 u_\no \\ \eta_0 u_{0} \\ \nabla (\eta_0 u_{0}) \end{bmatrix}\|_{L^2(\Omega,\C^{n+1},\widetilde\omega_\rho)}^2 \\
&\geq c^{-1} \kappa_B \left(
\int_{\mathbb{R}^n} (|\eta_0 u_\no|^2 + |\eta_0 u_{0}|^2) \widetilde\omega_\rho dx + 
\tfrac{1}{2} \int_{\mathbb{R}^n} |\eta_0 \nabla u_{0}|^2 \widetilde\omega_\rho dx - \int_{\mathbb{R}^n} |(\nabla \eta_0) u_{0}|^2 \widetilde\omega_\rho dx \right) \\
&\geq c^{-1} \kappa_B \left(
\tfrac{1}{2} \int_{\mathbb{R}^n}  \eta_0^2 |u|^2 \widetilde\omega_\rho dx - \int_{\mathbb{R}^n} |\nabla \eta_0|^2 |u_{0}|^2 \widetilde\omega_\rho dx \right),
\end{align*}  
for some positive constant $c<\infty$ only depending on $n$, due to the choice $C_{gbg}=4$.

Meanwhile, introducing a constant $\epsilon\in(0,1)$ to be chosen, and applying Cauchy's inequality with $\epsilon$, the remaining error term is controlled by
\begin{align*}
|E| &\leq c\|B\|_\infty \left(
2\int_{\mathbb{R}^n} |(\nabla \eta_0) u_{0}|  |\eta_0 u| \widetilde\omega_\rho dx
+ \int_{\mathbb{R}^n} |(\nabla \eta_0) u_{0}|^2 \widetilde\omega_\rho dx
\right) \\
&\leq c\|B\|_\infty \left(
2\epsilon
\int_{\mathbb{R}^n} \eta_0^2 |u|^2 \widetilde\omega_\rho dx
+\left(2\epsilon^{-1}+1\right)\int_{\mathbb{R}^n} |\nabla \eta_0|^2 |u_{0}|^2 \widetilde\omega_\rho dx
\right).
\end{align*}
Altogether, we have
\begin{align*}
&\re \langle B_{\rho,r} u, u \rangle= \re \langle B_\rho (\eta_0^2 u), u  \rangle+ \langle(1-\eta_0^2)u,u\rangle
+ r \|u_0\|^2
\\
&\geq c^{-1} \kappa_B \left(
\tfrac{1}{2} \int_{\mathbb{R}^n}  \eta_0^2 |u|^2 \widetilde\omega_\rho dx - \int_{\mathbb{R}^n} |\nabla \eta_0|^2 |u_{0}|^2 \widetilde\omega_\rho dx\right)
+\int_{\mathbb{R}^n} (1-\eta_0^2)|u|^2 \widetilde\omega_\rho dx
\\&\quad
- c\|B\|_\infty \left(
2\epsilon
\int_{\mathbb{R}^n} \eta_0^2 |u|^2 \widetilde\omega_\rho dx
+3\epsilon^{-1}\int_{\mathbb{R}^n} |\nabla \eta_0|^2 |u_{0}|^2 \widetilde\omega_\rho dx\right)
+ r \int_{\mathbb{R}^n} |u_0|^2 \widetilde\omega_\rho dx
\\
&= (\tfrac{1}{2} c^{-1}\kappa_B - c\|B\|_\infty 2\epsilon) \int_{\mathbb{R}^n}  \eta_0^2 |u|^2 \widetilde\omega_\rho dx
+\int_{\mathbb{R}^n} (1-\eta_0^2)|u|^2 \widetilde\omega_\rho dx
\\&\quad
-(c^{-1} \kappa_B + c\|B\|_\infty 3\epsilon^{-1}) \int_{\mathbb{R}^n} |\nabla \eta_0|^2 |u_{0}|^2 \widetilde\omega_\rho dx + r \int_{\mathbb{R}^n} |u_0|^2 \widetilde\omega_\rho dx.
\end{align*}
  We first choose $\epsilon>0$ sufficiently small, and then choose $r<\infty$ sufficiently large, depending only on $n$, $\kappa_B$ and $\|B\|_\infty$ such that
$$
\re \langle B_{\rho,r} u, u \rangle_{L^2(\R^n,\C^{2+n},\widetilde\omega_\rho)}
\geq 
\min\{\tfrac{1}{4} c^{-1}\kappa_B,1\} \|u\|_{L^2(\R^n,\C^{2+n},\widetilde\omega_\rho)}^2,
$$
so the result follows. 
\end{proof}

\subsection{Estimates on $\R^n$}   \label{sec:finalbits}

In this section, we complete the proof of Theorem~\ref{thm:main} 
by reducing it to, and proving, the corresponding weighted quadratic 
estimate for $D_\rho B_{\rho, r}$ in
$L^2(\mathbb{R}^n,\C^{n+1},\widetilde\omega_\rho)$.

\begin{Thm}\label{thm:mainRn}
We have quadratic estimates
\[
\int_0^\infty \|\psi_t(D_\rho B_{\rho,r}) u\|_{L^2(\mathbb{R}^n,\C^{2+n},\widetilde\omega_\rho)}^2 \frac{dt}{t} \lesssim \|u\|_{L^2(\mathbb{R}^n,\C^{2+n},\widetilde\omega_\rho)}^2, \qquad  u\in L^2(\mathbb{R}^n,\mathbb{C}^{2+n},\widetilde\omega_\rho),
\]
where the implicit constant depends only on $n$, $\|B_{\rho,r}\|_\infty$, $\kappa_{B_{\rho,r}}$ and 
$[\widetilde\omega_\rho]_{A_2(\mathbb{R}^n)}$. 
\end{Thm}

\begin{proof}
  The proof follows closely that presented in \cite[Sec.~3]{ARR},
  and we only point out the differences which arise due to the 
  presence of zero order identity blocks in $D_\rho$.
  Since $\ran(D_\rho)$ is closed, it suffices to bound the integral
  $\int_0^1$.
  As in the proof of \cite[Prop.~3.7]{ARR}, we approximate
  $Q_t^B= \psi_t(D_\rho B_{\rho,r})$ with a dyadic paraproduct
  $\gamma_t E_t$ by writing
$$
  Q_t^B u= Q_t^B(I-P_t)u+
  (Q_t^B-\gamma_t E_t)P_t u+ 
  \gamma_t E_t(P_t-I)u+ \gamma_t E_t u.
$$
By \cite[Prop.~3.1(i)]{ARR}, we may assume that
$u\in\ran(D_\rho)$.
Note that $\ran(D_\rho)$ is now a closed subspace, as compared
to \cite{ARR}.
As in \cite[Defn.~3.5]{ARR}, the principal part $\gamma_t$ of $Q_t^B$
is defined to be the multiplication operator 
$\gamma_t(x) z= (Q_t^B z)(x)$, $x\in\R^n$, where we view
$z\in\C^{2+n}$ as a constant function on $\R^n$.

The mollification operators $P_t$ and $E_t$, on scale $t$, are
$$
  P_tu= \begin{bmatrix}
    (I-t^2\Delta_{\widetilde \omega_\rho})^{-1}u_\no \\
    (I-t^2 \Delta)^{-1} u_0 \\
    (I-t^2 \Delta)^{-1} u_\ta
  \end{bmatrix},
$$ 
where $\Delta_{\widetilde \omega_\rho}=\div_{\widetilde \omega_\rho}
\nabla$, and with $\Delta=\div\nabla$ acting componentwise,
and
$$
  E_tu(x)= \begin{bmatrix}
   \barint_Q u_\no \widetilde \omega_\rho dx \\
   \barint_Q   u_0 dx  \\
   \barint_Q   u_\ta dx
  \end{bmatrix},
$$ 
where $Q$ denotes the dyadic cube on scale $t$ which contains 
$x$. See \cite[Sec.~2.2]{ARR}.
Note the new $u_0$ component as compared to
\cite{ARR}.

The first term $Q_t^B(I-P_t)u$, assuming $u\in\ran(D_\rho)$,
is estimated as in \cite[Lem.~3.10]{ARR}, now using the identity  
$$
  (I-P_t) u=
  tD_\rho
  \begin{bmatrix}
    t(-\Delta)^{1/2} (I-t^2 \Delta)^{-1} (-\Delta)^{1/2} u_0 \\
    0 \\
    t\nabla(I-t^2 \Delta_{\widetilde \omega_\rho})^{-1} u_\no
  \end{bmatrix},
  \qquad u=
  \begin{bmatrix}
    u_\no \\
    u_0 \\
    \nabla u_0
  \end{bmatrix}.
$$
The second term $(Q_t^B-\gamma_t E_t)P_t u$
is estimated as in \cite[Lem.~3.11]{ARR}, using polynomial off-diagonal
estimates for $D_\rho B_{\rho,r}$ on $\R^n$.
The third term $\gamma_t E_t(P_t-I)u$ is estimated as in 
\cite[Lem.~3.12-14]{ARR}, with the new $u_0$ component 
estimated in the same way as $u_\ta$.

The last term $\gamma_t E_t u$ is estimated using the double 
stopping time argument in \cite[Sec.~3.5-7]{ARR}, and this estimate
requires that $t<1$ (which we may assume since $\ran(D_\rho)$ is closed).
The only adaption needed in this proof is in 
\cite[Lem.~3.16]{ARR}, where now
$$
  E_{Q_1}(f_{Q_1}^\xi)-\xi= 
  \begin{bmatrix}
   \barint_Q (u_0-\div_{\widetilde\omega_\rho}u_\ta) \widetilde \omega_\rho dx \\
   \barint_Q   u_\no dx  \\
   \barint_Q   \nabla u_\no dx
  \end{bmatrix},
$$
using the notation of \cite[Lem.~3.16]{ARR}.
Besides \cite[Lem.~3.12-13]{ARR}, we require the following 
two estimates.
For $u_0$, we have 
\begin{multline*}
  \left| \barint_{Q_1} u_0 \widetilde \omega_\rho dx \right|
  \le  \left( \barint_{Q_1} |u_0|^2 \widetilde \omega_\rho dx \right)^{1/2}
 \\ \lesssim \sigma_3\ell(Q_1)
  \left( \int_{Q_1}\widetilde \omega_\rho dx\right)^{-1/2} 
   \left(\int_{\R^n} |1_{Q_1}\xi|^2 \widetilde \omega_\rho dx \right)^{1/2}
   \lesssim \sigma_3,
\end{multline*}
provided $\ell(Q_1)\lesssim 1$.
For $u_\no$, we have 
\begin{multline*}
  \left| \barint_{Q_1} u_\no dx \right|
  \le \frac 1{|Q_1|} \left( \int_{Q_1} |u_\no|^2 \widetilde \omega_\rho dx \right)^{1/2}
  \left( \int_{Q_1} \widetilde \omega_\rho^{-1} dx \right)^{1/2}
 \\ \lesssim \left( \barint_{Q_1} |u_\no|^2 \widetilde \omega_\rho dx \right)^{1/2}\lesssim \sigma_3,
\end{multline*}
using the $A_2$ condition on $\widetilde \omega_\rho$ and estimating
as for $u_0$ at the end.

Replacing $\C^{1+n}$ by $\C^{2+n}$, the rest of the argument from
 \cite[Sec.~3]{ARR}   goes through, and completes the proof.
\end{proof}

\begin{proof}[Proof of Theorem~\ref{thm:main}]
To compare $\psi_t(BD)$ and $\psi_t(D_\rho B_{\rho,r})$, we use 
a fixed smooth cut-off function $\eta:\R^n\rightarrow [0,1]$ such that
$$
 \sppt \eta \subset \Omega \quad \text{ and } \quad \eta\equiv1 \text{ on } 
\tfrac 12\Omega. 
$$
Denote the resolvents by
$R_t=(I+itDB)^{-1}$ and $R_t^{\rho,r}=(I+itD_\rho B_{\rho,r})^{-1}$.
Consider $u\in L^2(\mathcal{V},\omega)$ with $\sppt u \subseteq 4B_i$.
We have
\begin{align*}
R_t u &= R_t {T_\rho^*}^{-1} T_\rho^* ((\eta\circ\rho^{-1}) u)
= R_t {T_\rho^*}^{-1} \eta T_\rho^* u\\
&= [R_t {T_\rho^*}^{-1}\eta - {T_\rho^*}^{-1}\eta R_t^{\rho,r}]T_\rho^*u + {T_\rho^*}^{-1}\eta R_t^{\rho,r} T_\rho^*u\\
&= R_t {T_\rho^*}^{-1} [\eta (I+itD_\rho B_{\rho,r}) - {T_\rho^*}(I+itDB){T_\rho^*}^{-1}\eta] R_t^{\rho,r} T_\rho^*u + {T_\rho^*}^{-1}\eta R_t^{\rho,r} T_\rho^*u\\
&= R_t {T_\rho^*}^{-1} it(\eta D_\rho B_{\rho,r} - D_\rho B_\rho\eta) R_t^{\rho,r} T_\rho^*u + {T_\rho^*}^{-1}\eta R_t^{\rho,r} T_\rho^*u\\
&=R_t {T_\rho^*}^{-1}  it ([\eta,D_\rho] B_\rho \eta_0^2 + r\eta D_\rho \mathbb{P}_0) R_t^{\rho,r} T_\rho^*u + {T_\rho^*}^{-1}\eta R_t^{\rho,r} T_\rho^*u,
\end{align*}
where in the second last line we used Lemma~\ref{lem:locintertw} and
in the final line we used~\eqref{eq:BrhoR} 
to obtain
\begin{align}\begin{split}\label{eq:localisation}
\eta D_\rho B_{\rho,r} &- D_\rho B_\rho\eta \\
&= \eta D_\rho [B_\rho (\eta_0^2 I) + (1-\eta_0^2)I + r\mathbb{P}_0] - D_\rho B_\rho\eta \\
&= [\eta,D_\rho] B_\rho \eta_0^2 + D_\rho B_\rho\eta(\eta_0^2-1) + \eta D_\rho (1-\eta_0^2) + r \eta D_\rho \mathbb{P}_0 \\
&= [\eta,D_\rho] B_\rho \eta_0^2 + r\eta D_\rho\mathbb{P}_0,
\end{split}\end{align}
since $\eta_0=1$ on $\sppt\eta$.

We now note the identity $\psi_t(DB)=\frac{1}{2i}(R_{-t}-R_t)$,
the pointwise commutator estimate $|[\eta,D_\rho]u(x)| \leq |\nabla\eta(x)| |u(x)|$,
and that $D_\rho\mathbb{P}_0(u_\no,u_0,u_\ta)=(u_0,0,0)$
is a bounded operator.
This shows that
\[
\|\ca_{4B_i} \psi_t(DB) \ca_{4B_i} u\|_{L^2(\mathcal{V},\omega)}^2
\lesssim t^2 \|\ca_{4B_i} u\|_{L^2(\mathcal{V},\omega)}^2
+ \|\psi_t(D_\rho B_{\rho,r})T_\rho^*(\ca_{4B_i} u)\|_{L^2(\mathbb{R}^n,\C^{2+n},\widetilde\omega_\rho)}^2
\]
for all $t\in(0,T]$, all $u \in L^2(\mathcal{V},\omega)$ and all $i\in\N$. Theorem~\ref{thm:mainRn} therefore provides the estimate
\begin{align*}
I_1 &\lesssim \sum_{i\in\N} \left( \int_0^T t \|\ca_{4B_i} u\|_{L^2(\mathcal{V},\omega)}^2 \  dt + \int_0^T  \|\psi_t(D_\rho B_{\rho,r})(T_\rho^*\ca_{4B_i} u)\|_{L^2(\mathbb{R}^n,\C^{2+n},\widetilde\omega_\rho)}^2 \frac{dt}{t} \right) \\
&\lesssim  \sum_{i\in\N} \left(\|\ca_{4B_i} u\|_{L^2(\mathcal{V},\omega)}^2 + \|T_\rho^*\ca_{4B_i} u\|_{L^2(\mathbb{R}^n,\C^{2+n},\widetilde\omega_\rho)}^2\right) \\
&\lesssim \|u\|_{L^2(\mathcal{V},\omega)}^2,
\end{align*}
which completes the proof of Theorem~\ref{thm:main}.
\end{proof}

\begin{proof}[Proof of Theorem~\ref{thm:Kato}]
  In the splitting
  $$
    L^2(\mathcal{V},\omega)=
    L^2(M,\omega)\oplus L^2((M\times \C)\oplus T M,\omega),
  $$
  we consider operators
  $$
    D=\begin{bmatrix} 0 & S^* \\ S & 0 \end{bmatrix}
    \quad\text{and}\quad
    B=\begin{bmatrix} a & 0 \\  0 & \tilde A \end{bmatrix},
  $$
  where 
  $S=\begin{bmatrix} I \\ \nabla_M \end{bmatrix}$
  and
  $\tilde A=\begin{bmatrix} d & c \\  b & A \end{bmatrix}$.
  It follows from the accretivity hypotheses that $B$ is accretive on $\ran(D)$.
  Moreover,   
  $\ran(DB)=\ran(D)$ is closed since $\ran(S)$ is closed.
  Theorem~\ref{thm:main} proves quadratic estimates for $DB$.
  The $H^\infty$ functional calculus bounds then yield
  \begin{equation}  \label{eq:preKato}
    \|\sqrt{(BD)^2}f\|_{L^2(\mathcal{V},\omega)}\eqsim \|BDf\|_{L^2(\mathcal{V},\omega)}\eqsim \|Df\|_{L^2(\mathcal{V},\omega)},
  \end{equation}
  since the symbol $\lambda/\sqrt{\lambda^2}$ is self-inverse and 
  belongs to $H^\infty(S_\nu)$.
  Since $(BD)^2= \begin{bmatrix} aS^*\tilde A S & 0 \\ 0 & \tilde A SaS^*  \end{bmatrix}$, the Kato square root estimate in
   Theorem~\ref{thm:Kato} is the special case
  $f= \begin{bmatrix} u \\ 0 \end{bmatrix}$ of \eqref{eq:preKato}.
\end{proof}

\begin{Rem}
We remark that the full strength of the double stopping time argument
for weights and test functions 
from \cite[Sec.~3.5-7]{ARR}, that Theorem~\ref{thm:main} ultimately
builds on, is not  needed to prove Theorem~\ref{thm:Kato}. 
The reason is that Theorem~\ref{thm:Kato} only uses $B$ with 
a block form structure. See \cite[Sec.~3.8]{ARR}.
\end{Rem}

\section{Degenerate boundary value problems}  \label{sec:L2trace}

\subsection{The boundary $DB$ operator}   \label{sec:bdyop}

We consider second-order elliptic divergence form systems
\begin{equation}  \label{eq:origdivform}
  \divv A \nabla u=0
\end{equation}
on a compact manifold $\Omega$, with Lipschitz
boundary $\bdy \Omega$. 
We assume without loss of generality that $\Omega$ is connected, but not necessarily so for $\bdy \Omega$, 
thus allowing for geometries such as annuli. 
In general $u$ may take values in some $\C^N$, but to simplify the 
presentation we suppress this in notation.
As in \cite{ARR}, we only assume that the coefficents
are degenerate elliptic in the sense that 
$A/\omega \in L^\infty(\End(T\Omega))$ is accretive as in \eqref{eq:intaccr} and \eqref{eq:bdyaccr} below, for a given weight
$\omega\in A_2(\Omega)$.  
By $u$ being a weak solution to \eqref{eq:origdivform} we mean
that $u\in W^{1,2}_\loc(\Omega,\omega)$, so in particular $A\nabla u\in L^1_\loc(\Omega)$, and that 
$\divv(A\nabla u)=0$ holds in distributional sense. 

We let $M$ be a closed Riemannian manifold  (smooth in the sense of Section~\ref{sec:covering})  and consider the 
cylinder $\R\times M$.
We assume that there exists a bi-Lipschitz map
\begin{equation}  \label{eq:paramcyl}
  \rho: [0,\delta)\times M\to U\subset\Omega
\end{equation}
between a finite part of the cylinder $\R\times M$ 
and a neighbourhood $U$ of $\bdy \Omega$,
for some $\delta>0$, with $\rho(\{0\}\times M)=\bdy \Omega$
so that in particular $\partial \Omega$ is bi-Lipschitz homeomorphic to $M$.
Write $M_\rho= (0,\delta)\times M$. 
For example, in \cite{AR2} we used $\rho(t,x)= e^{-t}x$, $M$ being the unit sphere.  
As in Example~\ref{ex:pullcoeffs}, the function $u_\rho= u\circ \rho$ is a weak solution to the
elliptic divergence form system
\begin{equation}   \label{eq:secondonRxM}
   \divv A_\rho \nabla u_\rho=0
\end{equation}
on the cylinder $M_\rho$,
where $A_\rho= J_\rho (d\rho)^{-1} A (d\rho^*)^{-1}$.
In $M_\rho$, we assume that 
$$
\omega_\rho=\omega\circ \rho
$$ 
is independent of $t\in(0,\delta)$. 
We write $\omega_0= \omega_\rho|_M$ and note that 
$\omega_0\in A_2(M)$ by $t$-independence. 
This is natural since even in the uniformly elliptic case $\omega=1$, it is 
well known that some regularity of $t\to A_\rho$ is necessary for the results
that we aim at.
In $\Omega$, we use the measure $d\omega=\omega d\nu$ and
on $M$ we use the measure $d\omega_0= \omega_0d\mu$, where
$d\nu$ and $d\mu$ denote the Riemannian measures on $\Omega$ and $M$
respectively.

Following \cite{AA1, ARR}, 
we analyze $u_\rho$ on $M_\rho$ by
rewriting \eqref{eq:secondonRxM} as a generalized Cauchy--Riemann system
\begin{equation}   \label{eq:firstCR}
  \pd_t f + D_0B f=0, \qquad 0<t< \delta, x\in M,
\end{equation}
for the conormal gradient $f$ of $u_\rho$ at layer $t$, with
$D_0= \begin{bmatrix} 0 & \divv_{M,\omega_0} \\ -\nabla_M & 0 \end{bmatrix}$ having domain
$\Dom(D_0)= W^{1,2}(M,\omega_0) \oplus \Dom(\div_{M,\omega_0})$.
Writing $A_\rho= \begin{bmatrix} a & b \\ c & d \end{bmatrix}$
in the splitting $T(M_\rho)\simeq \C\times TM$, we define the conormal
gradient of $u_\rho$ to be
\begin{equation}   \label{eq:conormalgrad}
f=(f_t)_{0<t<\delta}= \begin{bmatrix} \omega_\rho^{-1}\pd_{\nu_{A_\rho}} u_\rho \\ \nabla_M u_\rho \end{bmatrix},
\end{equation} 
where $\pd_{\nu_{A_\rho}} u_\rho=a\pd_t u_\rho+ b \nabla_M u_\rho$
is the conormal derivative of $u_\rho$. 
We write 
$f_t$ for the function $f(t,\cdot)$ and sometimes, abusing notation,
also for the function $f$. 
Note that the boundary data $\pd_{\nu_{A_\rho}} u_\rho$ and
$\nabla_M u_\rho$ on $M$ essentially are the same as 
$\pd_{\nu_{A}} u$ and $\nabla_{\bdy\Omega} u$ on $\bdy \Omega$.
Indeed,  assuming that these boundary traces exist,  we have 
\begin{equation}   \label{eq:transfofCdata}
  \pd_{\nu_{A_\rho}} u_\rho= J_{\rho|_M}\pd_{\nu_{A}} u
  \circ\rho
  \quad\text{and}\quad \nabla_M u_\rho= (d\rho|_M)^*\nabla_{\bdy\Omega} u
  \circ\rho,
\end{equation}
where $J_{\rho|_M}$ and $(d\rho|_M)^*$ are invertible.
The transformed accretive coefficients $B=(B_t)_{0<t<\delta}$ 
appearing in \eqref{eq:firstCR} are
\begin{equation}   \label{eq:defntranformedB}
B=  \begin{bmatrix} \omega_\rho a^{-1} & -a^{-1}b \\ ca^{-1} & \omega_\rho^{-1}(d-ca^{-1}b) \end{bmatrix}.
\end{equation}
We remark that $d-ca^{-1}b$ is a well known object in matrix theory
called {\em the Schur complement} of $a$ in the matrix $A_\rho$,
and is denoted by $A_\rho/a$.
All this follows as in \cite[Sec.~4]{AA1}, replacing $(0,\infty)$ and $\R^n$
by $(0,\delta)$ and $M$, with modifications as in \cite{ARR}
in the degenerate case.
In particular, the scaling of $f$ by $\omega_\rho$ means that weak solutions 
$u\in W^{1,2}(\Omega, \omega)$ have conormal gradients
$f\in L^2_\loc((0,\delta),\rev L^2)$ on $M_\rho$, where
$\rev L^2=L^2(M, \omega_0)\oplus  L^2(TM, \omega_0)$.
Below in this section, $\|f\|= (\int_M|f|^2 d\omega_0)^{1/2}$ always refer to this weighted 
$\rev L^2$ norm, unless otherwise indicated. 
By $f\in L^2_\loc((0,\delta),\rev L^2)$ being a weak solution to 
\eqref{eq:firstCR}, we mean that this equation holds in distributional sense
on $M_\rho$. Note that $f$ and $Bf$ in particular belong to 
$L^1_\loc((M_\rho\times\C)\oplus TM_\rho)$. 

The accretivity assumptions on $A$ that we need
are
\begin{align}
   \re\int_\Omega \scl{A\nabla u}{\nabla u} d\nu &\gtrsim \int_\Omega |\nabla u|^2 d\omega,\quad u\in W^{1,2}(\Omega,\omega), \label{eq:intaccr} \\
   \int_M \scl{A_\rho(t,x)\nabla v(t,x)}{\nabla v(t,x)} d\mu(x) &\gtrsim\|\nabla v(t,x)\|^2,
   \quad  v\in W^{1,2}(M_\rho,\omega_\rho),  \label{eq:bdyaccr}
\end{align}
for almost every $t\in(0,\delta)$.
The second stronger assumption near $M$ is needed to obtain a 
bisectorial boundary operator.
In particular it 
ensures that the matrix element $a_t$ in $A_\rho$ is invertible, 
since it follows from \eqref{eq:bdyaccr} that $a_t$ is accretive on
all $L^2(M)$,  and that $B_t$
are bounded and accretive. 
See \cite[Sec.~3]{AR2} for this argument. 

The subspace
$$
\mH= \ran(D_0)= \ran(\divv_{M, \omega_0})\oplus \ran(\nabla_M)
\subset \rev L^2
$$
plays a central role in our analysis.

\begin{Prop}  \label{prop:cptresolvents}
The subspace $\mH$ is closed.
 Moreover, the restriction of the self-adjoint operator 
 $D_0: \Dom(D_0)\subset\rev L^2\to \rev L^2$
   to $\mH$ has a compact inverse $D_0^{-1}:\mH\to\mH$.
\end{Prop}

\begin{proof}
Recall from Lemma~\ref{lem:ddeladj} that $\divv_{M,\omega_0}= -(\nabla_M)^*$.
It also follows from weighted Poincar\'e inequalities in
\cite[Thm.~1.5]{FKS}
that $\nabla_M$ has a compact left Fredholm inverse.
In particular $\ran(\nabla_M)$ is closed and it follows from 
the closed range theorem, see \cite[Thm.~IV.5.13]{K}, that
$\ran(\divv_{M,\omega_0})$ is closed. 
Hence $\ran(D_0)$ is closed.
Since compactness is preserved when taking adjoints, 
this also shows that $D_0^{-1}:\mH\to\mH$ is compact.
\end{proof}

All the main estimates of the operators below build on the following 
quadratic estimate for the boundary operator $D_0B_0$, where $B_0= \lim_{t\to 0} B_t$.
The precise definition of this limit involves the Carleson norm 
\begin{equation}  \label{eq:modifiedCarl}
  \|\mE\|_*= \|C(W_\infty(|\mE|^2/t))\|_\infty^{1/2},
\end{equation}
where
$$
Cg(x)= \sup_{r<\delta} \left(\int_{B(x,r)} d\omega_0\right)^{-1}\int_{(0,r)\times B(x,r)} |g(t,y)|dtd\omega_0(y)
$$
and $(W_\infty g)(t,x)=\esssup_{W(t,x)} |g|$, using Whitney regions
$W(t,x)= (t/c_0,c_0t)\times B(x,c_1 t)$, for fixed $c_0>1, c_1>0$.
We assume that there exists $B_0\in L^\infty(M;\End(\C\times TM))$ such that 
$\|B_t-B_0\|_*<\infty$. As in \cite[Lem.~2.2]{AA1} (where $\|\cdot\|_C$
equals $\|\cdot\|_*$, but without the restriction $r<\delta$), 
such trace $B_0$ is unique determined by
$B=(B_t)_{0<t<\delta}$ and \eqref{eq:bdyaccr} implies the accretivity \eqref{eq:acc} for $B_0$, with $D$, $B$,
$(\mV,\omega)$ replaced by $D_0$, $B_0$, $(\C\times TM,\omega_0)$. 
Also let $A_0=\lim_{t\to 0} A_\rho$, in the sense that 
$\|A_\rho/\omega_\rho-A_0/\omega_0\|_*<\infty$.

\begin{Prop}   \label{prop:QEforbdyop}
For any bisectorial operator in $\rev L^2$ of the form $D_0B_0$,
where $D_0$ is the operator appearing  in \eqref{eq:firstCR} and $B_0$ is any bounded multiplication operator which is accretive on $\ran(D_0)$,  we have
$$
  \int_0^\infty \|tD_0B_0(I+t^2(D_0B_0)^2)^{-1} f\|^2 \frac{dt}t\lesssim \|f\|^2,
 \qquad \text{for all } f\in \rev L^2.
$$
\end{Prop}

\begin{proof}
  In the splitting
  $$
    L^2(M, \omega)\oplus 
    L^2(M, \omega)\oplus L^2(TM, \omega)
  $$
  we consider the operator $D$ from Theorem~\ref{thm:main}, given
  by \eqref{eq:Ddefn}.
  Given $B_0=\begin{bmatrix} a' & b' \\ c' & d' \end{bmatrix}$, we define
  $
     \widetilde B=\begin{bmatrix} a' & 0 & b' \\  0 & I & 0 \\ c' & 0 & d' \end{bmatrix}$.
  It follows that $\widetilde B$ is accretive on $\ran(D)$ in the sense \eqref{eq:acc}.
  Since $M$ is compact, the geometric hypotheses in
  Theorem~\ref{thm:main} hold, giving quadratic estimates
  for $D \widetilde B$.
  We have
$$
\begin{bmatrix} 0 & 0 & -\divv_{M,\omega_0} \\  0 & 0 & 0 \\ \nabla_M & 0 & 0 \end{bmatrix}
\begin{bmatrix} a' & 0 & b' \\  0 & 0 & 0 \\ c' & 0 & d' \end{bmatrix}
= D\widetilde B+ V
$$
for some bounded operator $V$. The left hand side is unitary equivalent to
$-D_0B_0\oplus 0$, so Proposition~\ref{prop:cptresolvents} shows that $\ran(D\widetilde B+V)$
is closed.
Lemma~\ref{lem:bddpert} now 
proves that quadratic estimates also hold for $D_0B_0$.
\end{proof}

As in Section~\ref{sec:funccal}, it follows from Proposition~\ref{prop:QEforbdyop}
that $D_0B_0$ has a bounded $H^\infty$ functional calculus   (the quadratic estimates for the adjoint $B_0^*D_0$ hold by similarity, since $B_0^*D_0=B_0^*(D_0B_0^*)(B_0^*)^{-1}$ on $B_0^*\ran(D_0)=\ran(B_0^*D_0)$).  
In particular we use the spectral projections
$E_0^+=\chi^+(D_0B_0)$ and $E_0^-=\chi^-(D_0B_0)$ for the right and left
half planes, given by the symbols 
$$
  \chi^\pm(\lambda)=
  \begin{cases}
     1, & \pm\re\lambda>0, \\
     0, & \pm \re\lambda\le 0. 
  \end{cases}
$$
The spectral subspaces $\ran(E^+_0)$ and $\ran(E_0^-)$ are closed complementary subspaces of
$\mH$, which generalize the classical Hardy subspaces from complex analysis, and
we denote them $E_0^\pm\mH$. 
We set $E^0_0= I-E^+_0-E^-_0$, which is the projection onto $\nul(D_0B_0)$ along $\mH$.
We also use the bounded Poisson semigroup 
$e^{-t\Lambda}$, with $\Lambda=|D_0B_0|$, which are the operators given by
the symbols $e^{-t|\lambda|}$, where
$|\lambda|= \lambda\sgn(\lambda)$ and
$\sgn(\lambda)= \chi^+(\lambda)-\chi^-(\lambda)$.

Similar to \cite[Prop.~3.3]{AR2} it follows that, for each fixed $0<t<\delta$, the conormal gradient $f$ of $u_\rho$ almost belongs to $\mH$.
The following lemma takes care of this technical issue.

\begin{Lem}   \label{lem:e0}
  Consider a weak solution $f=(f_t)_{0<t<\delta}\in L^2_\loc((0,\delta),\rev L^2)$ to 
  \eqref{eq:firstCR}. Then $E_0^0 f_t\in\rev L^2$ is independent of $t$ and belongs to a 
  finite dimensional subspace of $\rev L^2$.
  Denoting this constant value by $E_0^0f_0$, we have
  $$
     \|E^0_0f_0\|\simeq  \sum_i \Big|\int_{M_i}\pd_{\nu_{A_0}}u_\rho d\mu\Big|,
$$
where $M_i$ denote the connected components of $M$.
\end{Lem}

A simple example showing that $E^0_0f_0$ may be non-trivial is $u(t,x)=t$
on $\Omega= (0,1)\times S^1$.

\begin{proof}
  Applying $E_0^0$ to \eqref{eq:firstCR} immediately shows that $\pd_t E_0^0 f_t=0$.
  To see the norm estimate, we consider the orthogonal projection $P$ onto $\mH^\perp$
  in $\rev L^2$. Since $\nul(E_0^0)= \nul(P)= \mH$ it follows that 
  $\|E^0_0 f_t\|\simeq \|P f_t\|$.
  Indeed $\|E^0_0 f_t\|=\|E_0^0(P f_t+ (I-P)f_t)\|= \|E_0^0P f_t\|
  \lesssim \|Pf_t\|$, and the reverse estimate is proved similarly.
Since $\nabla_M u_\rho\in \ran(\nabla_M)$ and $\ran(\div_{M,\omega_0})^\perp= \nul(\nabla_M)$,
it follows that $Pf_t=\begin{bmatrix}c \\ 0\end{bmatrix}$, where $c$ is the locally constant function
given by the average of $\pd_{\nu_{A_\rho}}u_\rho$ on each component
$M_i$.
This shows the estimate of $E_0^0f_t= E_0^0 Pf_t$.
Since the locally constant functions $c$ form a finite dimensional space  (a basis is given by $\ca_{M_i}$) ,
this completes the proof.
\end{proof}

Besides the $\rev L^2$ topology for $f_0$, we also obtain results for $f_0$ in a larger
trace space $\rev W^{-1,2}$ which we now define.
We use test functions $v$ from the weighted Sobolev space $W^{1,2}(M,\omega_0)$ with
norm
$$
   \|v\|_{W^{1,2}}^2= \int_M(|\nabla v|^2 + |v|^2) d\omega_0.
$$
Define dual spaces 
\begin{align*}
W^{-1,2}(M,\omega_0^{-1}) &=(W^{1,2}(M,\omega_0))^*, \\
W^{-1,2}(TM,\omega_0)&=(W^{1,2}(TM,\omega_0^{-1}))^*.
\end{align*}
Using the unweighted $L^2(M,d\mu)$ pairing, we realize 
$L^2(M,\omega_0^{-1})$ as a subspace of $W^{-1,2}(M,\omega_0^{-1})$
and 
$L^2(TM,\omega_0)$ as a subspace of $W^{-1,2}(TM,\omega_0)$.
The weighted Sobolev $W^{-1,2}$ trace space that we use for $f$ is
\begin{equation}   \label{eq:Hminus1space}
  \rev W^{-1,2}= \omega_0^{-1}W^{-1,2}(M,\omega_0^{-1})
  \oplus W^{-1,2}(TM,\omega_0),
\end{equation}
that is $\pd_{\nu_A}u\in W^{-1,2}(M,\omega_0^{-1})$
and $\nabla_M u\in W^{-1,2}(TM,\omega_0)$.
Here $\omega_0^{-1}W^{-1,2}(M,\omega_0^{-1})$ denotes the
space dual to $\omega_0 W^{1,2}(M,\omega_0)$.
Note that the $W^{1,2}$ and $W^{-1,2}$ analogues of $\omega_0^{-1}L^2(M,\omega_0^{-1})=
L^2(M,\omega_0)$ do not hold since in general $\omega_0$ is not differentiable.

\begin{Lem}   \label{lem:HtoHminone}
  The Sobolev space $\rev W^{-1,2}$ contains  $\rev L^2$ as a dense subspace.
  Denote by $\mH^{-1}$ the closure of $\mH$ in $\rev W^{-1,2}$.
  The operators $D_0$, $D_0B_0$ and $\Lambda$ all have $\rev L_2$ domains
  which are dense in $\mH$. They all extend uniquely to isomorphisms $\mH\to \mH^{-1}$.  
\end{Lem}

\begin{proof} 
The denseness follows from a localization and mollification argument as in Lemma~\ref{lem:ddeladj}, and duality.
To prove the mapping properties $\mH\to \mH^{-1}$, 
consider first the operator $D_0$. 

We have a  semi-Fredholm map
$\nabla_M: W^{1,2}(M,\omega_0)\to L^2(TM,\omega_0)$, by
weighted Poincar\'e inequalities \cite[Thm.~1.5]{FKS}.
By duality this yields a semi-Fredholm map
$L^2(TM, \omega_0)\to \omega_0^{-1} W^{-1,2}(M, \omega_0^{-1})$, 
which by Lemma~\ref{lem:ddeladj} is seen to coincide with $\divv_{M, \omega_0}$
on the   dense subspace   $\Dom(\divv_{M, \omega_0})\subset L^2(TM, \omega_0)$.
We continue to write $\divv_{M, \omega_0}$ for this unique continuous
extension to $L^2(TM, \omega_0)$.

Similarly we have a bounded operator
$\divv_M: W^{1,2}(TM,\omega_0^{-1})\to L^2(M,\omega_0^{-1})$.
This is also a semi-Fredholm map, which
can be proved through weighted estimates of regularized Poincar\'e
maps for the divergence. See \cite[Sec.~10.4]{R2}
for the unweighted case, which can be extended to the weighted case,
similar to \cite[Thm.~1.5]{FKS}.
By duality this yields a semi-Fredholm map
$L^2(M,\omega_0)\to W^{-1,2}(TM,\omega_0)$,
which by Lemma~\ref{lem:ddeladj} is seen to coincide with $\nabla_M$
on the   dense subspace   $W^{1,2}(M,\omega_0)\subset L^2(M, \omega_0)$.
We continue to write $\nabla_M$ for this unique continuous
extension to $L^2(M, \omega_0)$.

This yields a unique continuous extension of $D_0$ to a semi-Fredholm
operator $\rev L^2\to \rev W^{-1,2}$, with null space 
$\mH^\perp\subset \rev L^2$ and range being the closure of $\mH$
in $\rev W^{-1,2}$.
Hence the restriction $D_0: \mH\to \mH^{-1}$ is an isomorphism.
For the operator $D_0B_0$ we note that the accretivity of $B_0$ on $\mH$
shows that $PB_0: \mH\to\mH$ is an isomorphism, where $P$ denotes
the orthogonal projection $\rev L^2\to \mH$.
Therefore $D_0B_0= D_0(PB_0): \mH\to\mH^{-1}$ is an isomorphism.
For the operator $\Lambda= D_0B_0\sgn(D_0B_0)$ we note that
$\sgn(D_0B_0):\mH\to \mH$ is an isomorphism, which yields
the isomorphism $\Lambda: \mH\to \mH^{-1}$.
\end{proof}

We want to define $b(D_0B_0)$ in the functional calculus
of $D_0B_0$ as operators on the closed subspace $\mH^{-1}$
and to show that these are bounded in the $\rev W^{-1,2}$ norm.
Using Lemma~\ref{lem:HtoHminone}, we extend operators $b(D_0B_0): \mH\to \mH$
in the functional calculus of $D_0B_0$, for $b\in H^\infty(S^o_\theta)$, 
to bounded operators 
$b(D_0B_0): \mH^{-1}\to \mH^{-1}$, by
setting 
$$
  b(D_0B_0)f= \Lambda(b(D_0B_0) \Lambda^{-1} f), \qquad f\in \mH^{-1}.
$$
($\Lambda$ can be replaced by $D_0B_0$ in this definition.)
It appears problematic to extend the operators $b(D_0B_0)$ to all of 
$\rev W^{-1,2}$, but we shall not need to do so.
We only need to act with $E_0^\pm$ on solutions $f_0\in \rev W^{-1,2}$ to \eqref{eq:firstCR}.
To this end, we recall from Lemma~\ref{lem:e0} that $E^0_0f_0\in \rev L^2$,
even if $f_0\in \rev W^{-1,2}$, since $t\mapsto E^0_0f_t$ is constant.
With this in mind, 
$$
  E^\pm_0 f_0 = E^\pm_0(f_0-E_0^0 f_0)
$$
is well defined, since $f_0-E_0^0 f_0\in \mH^{-1}$ and $E_0^\pm E_0^0=0$.

\subsection{Atiyah-Patodi-Singer conditions}   \label{sec:aps}

Given a weak solution $u$ to $\div A\nabla u=0$ in $\Omega$, with conormal
gradient $f$ on $M_\rho$, we prove in this section
existence and estimates of a trace $f_0$ on $M$ as $t\to 0$.
The spectral projections for the bisectorial operator $D_0B_0$ give the
splitting
\begin{equation}   \label{eq:f0split}
  f_0= E^+_0 f_0 + E_0^0 f_0+ E_0^- f_0.
\end{equation}
The part $E_0^+ f_0$ constitutes the APS boundary datum in our setting,
and we prove below equivalences of norms between the APS datum $E_0^+ f_0$
and the solution $\nabla u$,
modulo the part $E_0^0 f_0$, which takes values in a finite dimensional space.
We summarize this in the next two statements.

For boundary $L^2$ topology for $f_0$, we use the following 
non-tangential maximal function for $\nabla u$.
Given a function $f$ on the cylinder, supported on $M_\rho$, we define
the $L^2$ Whitney averaged non-tangential maximal function
$$
\widetilde N_* f(x)= \sup_{0<t<c_0\delta} (W_2f)(t,x),\qquad x\in M,
$$
where 
$(W_2f)(t,x)=\left(\barint_{W(t,x)}|f|^2 ds d\omega_0\right)^{1/2}$.

Let also $\eta_t= \max(0,\min(1,2-2t/\delta))$.
With slight abuse of notation we also consider $\eta$ as a function on $M_\rho$,
and on $\Omega$ by zero extension.
Our APS solvability estimate for the $\rev L^2$ boundary function space is the following.
 
\begin{Thm}   \label{thm:N}
Let $\Omega$, $\rho$, $M$, $\omega$, $A$ and $B_t$ be as in Section~\ref{sec:bdyop}.
Assume that there exists $t$-independent coefficients $B_0=B_0(x)$
on $M_\rho$,
so that $\|B_t-B_0\|_*<\infty$.
Consider a weak solution $u$ to \eqref{eq:origdivform}, 
with conormal gradient $f$ of $u_\rho$ as in \eqref{eq:conormalgrad}, where
$\nabla u$ belongs to the Banach space $\mX\subset L^2(T\Omega,\omega)$ 
with norm
\begin{equation}   \label{eq:Xnormdef}
  \|g\|^2_\mX= \int_M |\widetilde N_* (\eta g\circ \rho))|^2\omega_0d\mu + \int_\Omega |(1-\eta) g|^2 \omega d\nu.
\end{equation}
Then the trace $\nabla u|_{\bdy \Omega}$ exists in the sense that
\begin{equation}   \label{eq:Diniconv}
  \lim_{t\to 0}t^{-1}\int_t^{2t} \|f_s-f_0\|^2 ds=0,
\end{equation}
for some $f_0\in \rev L^2$ with 
\begin{equation}   \label{eq:Xest}
  \|f_0\|\lesssim \|\nabla u\|_\mX.
\end{equation}
If $\|B_t-B_0\|_*$ is small enough, for some $\delta>0$, 
depending on $[\omega_0]_{A_2(M)}$, $\|A|_M/\omega_0\|_\infty$ and the
constant in \eqref{eq:bdyaccr},
then we also have the estimate
\begin{equation}   \label{eq:Xestrev}
  \|\nabla u\|_\mX\lesssim \|E_0^+ f_0\|+ \|E_0^0 f_0\|,
\end{equation}
with the implicit constant also depending on the constant in 
\eqref{eq:intaccr}.
\end{Thm}

Our APS solvability estimate for the $\rev W^{-1,2}$ boundary function space is the following.

\begin{Thm}   \label{thm:D}
Let $\Omega$, $\rho$, $M$, $\omega$, $A$ and $B_t$ be as in Section~\ref{sec:bdyop}.
Assume that there exists $t$-independent coefficients $B_0=B_0(x)$
on $M_\rho$,
so that $\|B_t-B_0\|_*<\infty$.
Consider a weak solution $u$ to \eqref{eq:origdivform}, 
with conormal gradient $f$ of $u_\rho$  as in \eqref{eq:conormalgrad}, where
$\nabla u$ belongs to the Hilbert space $\mY\supset L^2(T\Omega,\omega)$,
with norm
\begin{equation}   \label{eq:Ynormdef}
  \|g\|^2_\mY= \int_M\int_0^\delta |\eta_t g(\rho(t,x))|^2\, tdt\,\omega_0 d\mu(x)
  +
  \int_\Omega |(1-\eta) g|^2 \omega d\nu.
\end{equation}
Then the trace $\nabla u|_{\bdy \Omega}$ exists in the sense that
\begin{equation}   \label{eq:Wminoneconv}
   \lim_{t\to 0}\|f_t-f_0\|_{\rev W^{-1,2}}=0,
\end{equation}
for some $f_0\in \rev W^{-1,2}$ with 
\begin{equation}   \label{eq:Yest}
  \|f_0\|_{\rev W^{-1,2}}\lesssim \|\nabla u\|_\mY.
\end{equation}
If $\|B_t-B_0\|_*$ is small enough, for some $\delta>0$, 
depending on $[\omega_0]_{A_2(M)}$, $\|A|_M/\omega_0\|_\infty$ and the
constant in \eqref{eq:bdyaccr}, then
then we also have the estimate
\begin{equation}   \label{eq:Yestrev}
  \|\nabla u\|_\mY\lesssim \|E_0^+ f_0\|_{\rev W^{-1,2}}+ \|E_0^0 f_0\|_{\rev W^{-1,2}},
\end{equation}
with the implicit constant also depending on the constant in 
\eqref{eq:intaccr}.
\end{Thm}

We recall that in both Theorems~\ref{thm:N} and \ref{thm:D}, the norm
of the finite dimensional part $E_0^0 f_0$ is equivalent to the sum
of the integrals of the conormal derivative over the connected components
of $M$, by Lemma~\ref{lem:e0}.

The main computation for the proofs of Theorem~\ref{thm:N} and \ref{thm:D}
is the following.
Consider a conormal gradient $f\in L^2_\loc((0,\delta), \rev L^2)$
solving \eqref{eq:firstCR}, which we multiply by $\eta_t$ and rewrite as
\begin{equation}   \label{eq:preintegration}
  (\pd_t + D_0B_0)(\eta_t f_t)= D_0B_0\mE_t\eta_t f_t+ \eta_t' f_t, \qquad
  0<t<\delta,
\end{equation}
with the Carleson multiplier $\mE_t= B_0^{-1}(B_0-B_t)$.
(Our accretivity assumptions do in general only imply the existence of $B_0^{-1}$ on $\mH$.
This technical issue is fixed as in \cite[Eqn.~(22)]{AA1}.)
We integrate \eqref{eq:preintegration} using operators in the $H^\infty$ functional
calculus of $D_0B_0$, bounds of which are supplied by Proposition~\ref{prop:QEforbdyop}.
Splitting \eqref{eq:preintegration} with the projections $E_0^+$, $E^-_0$
and $E^0_0$, we get 
\begin{align*}
  (\pd_t + \Lambda)(\eta_t E^+_0 f_t) &= \Lambda E_0^+\mE_t\eta_t f_t+ \eta_t' E^+_0 f_t, \\
   (\pd_t - \Lambda)(\eta_t E^-_0 f_t) &= -\Lambda E_0^-\mE_t\eta_t f_t+ \eta_t' E^-_0 f_t, \\
   \pd_t(\eta_t E^0_0 f_t) &= \eta_t' E^0_0 f_t.
\end{align*}
Integrating each of these equation and adding the obtained 
integral equations yields
\begin{equation}   \label{eq:postintegration}
  \eta_t f_t= e^{-t\Lambda} h^+ + S(\mE_t \eta_t f_t) + \widetilde S(\eta_t' f_t)+\eta_t E^0_0f_t, \qquad
  0<t<\delta.
\end{equation}
See the proof of \cite[Thm.~8.2]{AA1} for the detailed calculations
for the $E_0^\pm$ equations. In view of Lemma~\ref{lem:e0}, the addition of the $E_0^0$ equation is a trivial modification.

In the third term, $\eta' f$ is supported on $\delta/2\le t\le \delta$, and we
show below compactness properties of the operator
$$
  \widetilde S g_t= \int_0^t e^{-(t-s)\Lambda} E_0^+ g_s ds
  - \int_t^\delta e^{-(s-t)\Lambda} E_0^- g_s ds, \qquad
  0<t<\delta.
$$
The second term accounts for the $t$-variation of the coefficents and
uses the operator
$$
  Sg_t= \int_0^t \Lambda e^{-(t-s)\Lambda} E_0^+ g_s ds
  + \int_t^\delta \Lambda e^{-(s-t)\Lambda} E_0^- g_s ds,
  \qquad
  0<t<\delta.
$$
As in \cite[Prop.~7.1]{AA1}, Proposition~\ref{prop:QEforbdyop} implies 
boundedness of $S\mE$ on both $\mX$ and $\mY$, with norm bounded
by $\|\mE\|_*$.
In the first term, we obtain the existence of the boundary function $h^+$ on
$M$, belong to the range of $E_0^+$, by a limiting argument in the equation \eqref{eq:postintegration}.
See \cite[Thm.~8.2]{AA1} for the space $\mX$ and 
\cite[Thm.~9.2]{AA1} for the space $\mY$.
In the fourth term we recall that $E^0_0 f_t=E^0_0 f_0$ is independent of
$t$ and represents a finite rank operator, by Lemma~\ref{lem:e0}.

Letting $t=0$ in \eqref{eq:postintegration}, we note the formula
\begin{equation}   \label{eq:tracesplitrepr}
    f_0= h^+ + E_0^-\int_0^\delta e^{-s\Lambda} (\Lambda\mE_s\eta_s f_s-
    \eta'_s f_s) ds+ E_0^0 f_0
\end{equation}
for the full Cauchy trace $f_0$. Since $h^+\in\Ran(E_0^+)$, this
gives the splitting \eqref{eq:f0split} of $f_0$.

\begin{proof}[Proof of Theorem~\ref{thm:N}]
(i)  Consider a weak solution $u$ with $\nabla u\in\mX$, so that $\|\widetilde N_*(\eta f)\|<\infty$.
  To prove the existence of the trace $f_0$, we use \eqref{eq:postintegration}.
  The first and third terms converge in $\rev L^2$ as $t\to 0$, as in
  \cite[Prop.~6.4]{AA1}.
  The second term is seen to converge in the average $L^2$ sense
  \eqref{eq:Diniconv}
  as in part (iii) of the proof of \cite[Thm.~8.2]{AA1}.
  The estimate $\|f_0\|\lesssim \|\nabla u\|_\mX$ follows as in that
  proof.
  We conclude that there is a well defined trace $f_0$ satisfying
  \eqref{eq:tracesplitrepr} in $\rev L^2$.
  
  To prove the reverse estimate \eqref{eq:Xestrev}, 
  by the Open Mapping Theorem it suffices 
  to show that $\nabla u\mapsto (E_0^+ +E_0^0)f_0$ is an injective semi-Fredholm
  operator, as a bounded map from the closed subspace of $\mX$ consisting
  of gradients of solutions, to $\rev L^2$, which we prove next.
  
(ii) To prove injectivity, we define a cutoff function
$\hat\eta_\epsilon= \max(0,\min(1,t/\epsilon-1))$ on $M_\rho$
and extend by $\eta_\epsilon=1$ to the interior of $\Omega$.
For a weak solution to $\div A \nabla u=0$, we have
$
 \int_\Omega\scl{A\nabla u}{\nabla u}\hat\eta_\epsilon d\nu
  = -\epsilon^{-1} \int_\epsilon^{2\epsilon}\int_{M} (\pd_{\nu_{A_\rho}} u_\rho) \rev u_\rho d\mu dt.$
Letting $\epsilon\to 0$ and taking real parts, accretivity \eqref{eq:intaccr} shows that
\begin{equation}    \label{eq:Green}
 \int_\Omega |\nabla u|^2 d\omega\lesssim
  -\re\int_M (\pd_{\nu_{A_0}}u_\rho) \rev u_0 d\mu
\end{equation}
if $f_0= \begin{bmatrix}\omega_0^{-1}\pd_{\nu_{A_0}}u_\rho|_M \\ \nabla_M u_0\end{bmatrix}$.
This uses \eqref{eq:Diniconv} (which we established in (i) above) and $C([0,\delta),\rev L^2)$ continuity
for $u_t$, which is proved as in \cite[Prop.~7.2]{AA1}.

Assuming that $E_0^+f_0=0$ and $E_0^0 f_0=0$, 
we have $f_0= E_0^-f_0$.
Defining $f_t= e^{-t\Lambda}f_0$ for $t<0$, it follows that $\pd_t f_t+ D_0B_0 f_t =0$ for $t<0$
and that $f_t$ is the conormal gradient of a solution $u^-$ on $(-\infty, 0)\times M$,
with $L^2$ limit $f_0$ as $t\to 0^-$.
By suitable choices of constants of integration, for each connected component $M_i$, we may also assume
that $u^-= u_0$ at $t=0$.
It follows from Proposition~\ref{prop:cptresolvents} that 
$\|f_t\|\lesssim |t|^{-N}$ for any $N<\infty$, as $t\to -\infty$.
Another application of the Divergence Theorem shows that
\begin{equation}    \label{eq:extGreen}
  \re\int_M (\pd_{\nu_{A_0}} u^-) \rev{u^-} d\mu= \re\int_{-\infty}^0\int_M
  \scl{A_0\nabla u^-}{\nabla u^-} d\mu dt\ge 0,
\end{equation}
where the inequality follows from accretivity \eqref{eq:bdyaccr}.
Matching the right hand side in \eqref{eq:Green} and the
left hand side in \eqref{eq:extGreen}, the continuity at $t=0$ shows that
$\int_\Omega |\nabla u|^2 d\omega\le 0$, and
it follows that $\nabla u=0$ in $\Omega$.

(iii) To show that the range of $\nabla u\mapsto (E_0^+ +E_0^0)f_0$ 
is closed, we note
from \eqref{eq:postintegration} that
\begin{equation}  \label{eq:mainXest}
 \|1_{(0,\delta)}(I-S\mE_t)(\eta_t f_t)\|_\mX \le
 \|1_{(0,\delta)}e^{-t\Lambda}h^+\|_\mX + 
 \|1_{(0,\delta)}\widetilde S(\eta_t' f_t)\|_\mX
 +\|\eta_t E^0_0f_t\|_\mX.
\end{equation}
If $\|\mE\|_*$ is small enough, then the left hand side is 
$\gtrsim \|\eta_t f_t\|_\mX$.
Weighted non-tangential maximal estimates, proved as in 
\cite[Thm.~5.2]{AA1} but using weighted estimates 
\cite[Lem.~6.1]{ARR} and Proposition~\ref{prop:QEforbdyop}, 
shows that the first
term on the right hand side is
$\lesssim \|h^+\|= \|E_0^+f_0\|$. 
The last term is a finite rank operator, as discussed above, 
and is $\lesssim \|E^0_0 f_0\|$.
For the second term, we claim that 
$$
  1_{(0,\delta)}\widetilde S\eta_t' = 1_{(0,\delta/4)}\widetilde S\eta_t' + 1_{(\delta/4,\delta)}\widetilde S\eta_t'
  =I+II 
$$
is a compact operator on $\mX$.
For II, it suffices to show that
$$
  f_t\mapsto \int_{\delta/4}^t e^{-(t-s)\Lambda} E_0^+ g_s ds
  - \int_t^\delta e^{-(s-t)\Lambda} E_0^- g_s ds
$$
is compact on $L^2((\delta/4,\delta), \mH)$.
This follows from Proposition~\ref{prop:cptresolvents} by operational calculus
as in \cite[Prop.~6.4]{AA1}, since the symbols
$F(\lambda): L^2((\delta/4,\delta), \C)\to L^2((\delta/4,\delta), \C)$
(replacing $\Lambda$ by $\lambda\in\C$) are Hilbert--Schmidt
operators on $L^2((\delta/4,\delta), \C)$ with 
$\lim_{\lambda\to\infty}\|F(\lambda)\|=0$.
For I, we use the support of $\eta_t'$ and compute
\begin{multline*}
  \int_{\delta/2}^\delta e^{-(s-t)\Lambda} E_0^- g_s ds
  \\= \int_{\delta/2}^\delta e^{-(s-t)\Lambda}(I-e^{-2t\Lambda}) E_0^- g_s ds
  + e^{-t\Lambda}\int_{\delta/2}^\delta e^{-s\Lambda} E_0^- g_s ds
  = I_1+I_2.
\end{multline*}
For the first term, since $\mX\supset L^2((0,\delta/4), t^{-1}dt,\mH)$, 
it suffices to show that 
$I_1: L^2((\delta/2,\delta),\mH)\to L^2((0,\delta/4), t^{-1}dt,\mH)$
is compact.
Estimating $|e^{-(s-t)\lambda}(1-e^{-2t\lambda})|
\lesssim t|\lambda|e^{-(\delta/4)\re\lambda}$, this follows
by operational calculus as for $II$.
For $I_2$ we write
$$
  I_2= e^{-t\Lambda}e^{-(\delta/4)\Lambda}
  \int_{\delta/2}^\delta  e^{-(s-\delta/4)\Lambda} E_0^- g_s ds,
$$
where the left factor is bounded $\mH\to \mX$,
the middle factor is compact $\mH\to \mH$
as a consequence of Proposition~\ref{prop:cptresolvents}, and the right factor
is bounded $L^2((\delta/2,\delta),\mH)\to \mH$.

(iv)
Summarizing, we obtain from \eqref{eq:mainXest}
the estimate
\begin{equation}   \label{eq:bdylowerbound}
  \|\eta_t f_t\|_\mX\lesssim \|E_0^+ f_0\|+ \|Kf_t\|_\mX
\end{equation}
with a compact operator $K$.
To estimate $\nabla u$ on the interior of $\Omega$, we use the 
Caccioppoli inequality
$$
  \int_\Omega |\nabla u|^2 \tilde \eta^2 d\omega \lesssim
   \int_\Omega |u|^2 |\nabla\tilde \eta|^2 d\omega,
$$
where
$\tilde\eta\in C^\infty_c(\Omega)$
with $\tilde \eta=1$ on $\Omega\setminus\rho((0,\delta/2)\times M)$.
Applying weighted Poincar\'e inequalities
\cite[Thm.~1.5]{FKS} to $u$ on $\supp\nabla\tilde \eta$,
and combining with \eqref{eq:bdylowerbound},
proves a lower estimate $\|\nabla u\|_\mX\lesssim \|E_0^+ f_0\|+ \|K(\nabla u)\|_\mX$,
with a compact operator $K$,
in the case when $\supp\nabla\tilde \eta$ is connected.
The case when $\supp\nabla\tilde \eta$ is disconnected does not cause
a problem, since the difference between the mean values of $u$
on different components can be added as a finite rank operator to $K$.
This implies that the range of $\nabla u\mapsto (E_0^++E_0^0) f_0$ is closed,
see \cite[Thm.~IV.5.26]{K},
and completes the proof of Theorem~\ref{thm:N}.
\end{proof}

\begin{proof}[Proof of Theorem~\ref{thm:D}]

(i) Consider a weak solution $u$ with $\nabla u\in\mY$, so that 
$\int_0^\delta\|\eta_t f_t\|^2 tdt <\infty$.
Similar to \cite[Thm.~8.2]{AA1}, we write the first term in
   \eqref{eq:postintegration} as
$$   
     e^{-t\Lambda} h^+= \Lambda e^{-t\Lambda} \tilde h^+,
$$
with $\tilde h^+\in E_0^+ \mH$ and 
$h^+=\Lambda \tilde h^+\in\mH^{-1}\subset \rev W^{-1,2}$,
using Lemma~\ref{lem:HtoHminone}. 
Existence of the trace $f_0$ 
and the estimate $\|f_0\|_{\rev W^{-1,2}}\lesssim \|\nabla u\|_\mY$
follows from strong convergence $e^{-t\Lambda}\to I$ on $\rev L^2$, Lemma~\ref{lem:HtoHminone}
and estimating as in \cite[Prop.~7.2]{AA1} for the second term in \eqref{eq:postintegration}.
Note that the estimate of the trace of $E_0^0 f_t$ is trivial,
since it is a locally constant function on $M_\rho$
and $\mY\subset L^2_\loc((0,\delta),\rev L^2)$. 
  We conclude that there is a well defined trace $f_0$ satisfying
  \eqref{eq:tracesplitrepr} in $\rev W^{-1,2}$.
  
  To prove the reverse estimate \eqref{eq:Yestrev}, it suffices 
  to show that $\nabla u\mapsto (E_0^++E_0^0)f_0$ is an injective semi-Fredholm
  operator, as a bounded map from the closed subspace of $\mY$ consisting
  of gradients of solutions, to $\rev W^{-1,2}$.
  
  (ii)
To prove injectivity, assume that $\nabla u\in \mY$ and $E_0^+f_0=0$, 
 $E_0^0 f_0=0$.
It follows from \eqref{eq:postintegration} that
$$
 1_{(0,\delta)}(I-S\mE_t)(\eta_t f_t)=1_{(0,\delta)}\widetilde S(\eta_t' f_t).
$$
It follows from step (iii) in the proof of Theorem~\ref{thm:N}
that in fact the right hand side belongs to $\mX$.
By hypothesis $1_{(0,\delta)}S\mE_t 1_{(0,\delta)}$
is bounded with small norm on both $\mY$ and on $\mX$, where 
$\mX\subset\mY$.
Using a Neumann series for the inverse shows that in fact $\nabla u\in \mX$,
and $\nabla u=0$ now follows as in step (ii) in the proof of 
Theorem~\ref{thm:N}.
  
(iii)
To prove that the range is closed, we note that
\begin{equation} 
 \|1_{(0,\delta)}(I-S\mE_t)(\eta_t f_t)\|_\mY \le
 \|1_{(0,\delta)}\Lambda e^{-t\Lambda}\tilde h^+\|_\mY + 
 \|1_{(0,\delta)}\widetilde S(\eta_t' f_t)\|_\mY
 +\|\eta_t E^0_0f_t\|_\mY.
\end{equation}
If $\|\mE\|_*$ is small enough, then the left hand side is 
$\gtrsim \|\eta_t f_t\|_\mY$.
Quadratic estimates, proved as in 
\cite[Thm.~5.2]{AA1} but building on Proposition~\ref{prop:QEforbdyop}, 
show that the first term on the right hand side is
$\lesssim \|\tilde h^+\|= \|E_0^+f_0\|_{\rev W^{-1,2}}$,
and the last term is a finite rank operator.
The second term maps
$L^2((\delta/2,\delta),\mH)\to L^2((0,\delta),\mH)\subset\mY$,
where the inclusion is continuous, and the $L^2$ map is shown
to be compact, arguing as for terms $II$ and $I_1$ in the 
proof of  Theorem~\ref{thm:N}.
Estimating $u$ on the interior of $\Omega$ as in part (iv) of that proof,
it follows that the range of $\nabla u\mapsto (E_0^++ E_0^0) f_0$ 
is closed.
\end{proof}

\begin{proof}[Proof of Theorem~\ref{thm:ND}]
The estimate \eqref{eq:Xest0} is the same as \eqref{eq:Xest},
which has been proved above.
The estimate \eqref{eq:Yest0} is the same as \eqref{eq:Yest},
which has been proved above.
In Theorem~\ref{thm:N}, the reverse estimate
$\|\nabla u\|_\mX\lesssim \|E_0^+ f_0\|+ \|E_0^0 f_0\|$ was also
proved.
In Theorem~\ref{thm:D}, the corresponding $\rev W^{-1,2}$ estimate
was shown.
It remains only to note that
$\|E_0^+ f_0\|+ \|E_0^0 f_0\|\lesssim \|f_0\|$, and the corresponding
estimate for $\rev W^{-1,2}$ norms, which hold due to the 
boundedness of the operators and Lemma~\ref{lem:e0}.
\end{proof}

\subsection{Neumann and Dirichlet conditions}   \label{sec:dirneu}

In this section we further assume that 
the coefficients $A_\rho$ have trace $A_0$ which  
is $L^\infty$ close to a self-adjoint coefficient matrix, 
and we prove Theorem~\ref{thm:dirneu}.
The main ingredient in the proof is the following Rellich identity.

\begin{Lem}  \label{lem:rellich}
   Consider the operator $D_0\widetilde B$ acting in $\rev L^2$ on
   $M$, where $\widetilde B\in L^\infty(\End(\C\oplus TM))$ is accretive
   on $\ran(D_0)$, see \eqref{eq:acc}, and $\widetilde B$ is the transform
  \eqref{eq:defntranformedB} of self-adjoint coefficients $\tilde A$.
  Then we have the Rellich identity
  $$
  \int_M \scl{h_\no}{(\widetilde B h)_\no} d\omega_0
  = \int_M \scl{h_\ta}{(\widetilde B h)_\ta} d\omega_0
  $$
  for all $h= \begin{bmatrix} h_\no \\ h_\ta \end{bmatrix}
  \in \chi^+(D_0\widetilde B)\rev L^2$.
  The same identity holds also for all 
  $h \in \chi^-(D_0\widetilde B)\rev L^2$.
\end{Lem}

\begin{proof}
  Introduce the auxiliary operator 
  $N=\begin{bmatrix} -1 & 0 \\ 0 & I \end{bmatrix}$ and
  note that $ND_0+D_0N=0$.
  Furthermore we observe that the self-adjointness of $\tilde A$
  translates to $\widetilde B^*= N\widetilde B N$ with
  \eqref{eq:defntranformedB}.
  Define the Cauchy extension $h_t= e^{-t\Lambda} h$, $t>0$, of 
  $h$ on the half-cylinder $(0,\infty)\times M$.
  Since   $h \in \chi^+(D_0\widetilde B)\rev L^2$,   we have
$$
  \pd_t h_t= - \Lambda h_t= -D_0\widetilde B h_t
$$  
and $h_t\to h$ in $\rev L^2$ as $t\to 0$.
Proposition~\ref{prop:cptresolvents} and the boundedness of the $H^\infty$ functional calculus
yields $\|h_t\|\lesssim\|(D_0\widetilde B)^Ne^{-t\Lambda} h\|
=t^N\|(tD_0\widetilde B)^Ne^{-t\Lambda} h\|\lesssim t^{-N}$ for any $N<\infty$
as $t\to\infty$.
We obtain
 \begin{align*}
\scl{Nh}{\widetilde Bh}_{\rev L^2}
& = -\int_0^\infty \pd_t \scl{Nh_t}{\widetilde Bh_t}_{\rev L^2} dt\\
& = \int_0^\infty \Big( \scl{N(D_0\widetilde Bh_t)}{\widetilde Bh_t}_{\rev L^2} + \scl{Nh_t}{\widetilde B(D_0\widetilde Bh_t)}_{\rev L^2} \Big) dt \\
& = \int_0^\infty \scl{(ND_0\widetilde B+D_0\widetilde B^*N)h_t}{\widetilde Bh_t}_{\rev L^2}  dt
=0,
\end{align*} 
from which the stated Rellich identity follows.

The proof for $h \in \chi^-(D_0\widetilde B)\rev L^2$ is similar,
using $h_t$ on the lower half cylinder $t<0$.
\end{proof}

\begin{Def}
Define the subspaces 
\begin{align*}
N^+\mH = \{0\} \oplus \ran(\nabla_M), \quad
N^-\mH = \ran(\div_{M,\omega_0}) \oplus \{0\}
\end{align*}
of tangential and normal vector fields on $M$, respectively,
with the projections
\[
N^+f=f_\ta,
\quad
N^-f= f_\no,
\quad
f= \begin{bmatrix}f_\no \\ f_\ta\end{bmatrix}.
\]
\end{Def}

Using Rellich identities, we prove that these subspaces are transversal
to the subspaces $E^\pm_0 \mH$.
More precisely, we have the following.

\begin{Prop}   \label{prop:restrictproj}
Let $\widetilde B$ be the transform of self-adjoint coefficients 
as in Lemma~\ref{lem:rellich}.   Then there exists $\epsilon>0$
such that whenever coefficients $B$ satisfy $\|B-\widetilde B\|_\infty<\epsilon$, 
the four restricted projections  
$$
E_0^\pm \mH\to N^\pm \mH: h\mapsto N^\pm h,
$$
as well as the four restricted projections
$$
  N^\pm \mH\to E_0^\pm \mH: h\mapsto E_0^\pm h,
$$
are all invertible isomorphisms. Here $E_0^\pm= \chi^\pm(D_0B)$.
The same hold if $\mH\subset \rev L^2$ is replaced by 
$\mH^{-1}\subset\rev W^{-1,2}$.
\end{Prop}

\begin{proof}
(i) 
 Consider first transforms of self-adjoint coefficients, that is, $B= \widetilde B$.
 Lemma~\ref{lem:rellich} shows that
  \begin{multline*}
    \|h\|^2\lesssim  
    \re \int_M \scl{h_\no}{(\widetilde B h)_\no} d\omega_0
  + \re \int_M \scl{h_\ta}{(\widetilde B h)_\ta} d\omega_0 \\
  \lesssim \|h_\no\| \|(\widetilde B h)_\no\|
  \lesssim \|h_\no\| \|h\|,
  \end{multline*}
  where the first estimate follows from the accretivity of $\widetilde B$.
  Therefore 
  $\|h\|\lesssim \|h_\no\|$
  for all $h\in E^\pm_0\mH$.
  A similar argument, instead keeping the $\scl{h_\ta}{(\widetilde B h)_\ta}$
  term, proves that $\|h\|\lesssim \|h_\ta\|$.
  This shows that $N^\pm: E_0^\pm \mH\to N^\pm \mH$ are 
  all injective semi-Fredholm operators.
  This entails that the four restricted projections
  $E_0^\pm: N^\pm \mH\to E_0^\pm \mH$ also are 
  injective semi-Fredholm operators.
  We prove this for $E_0^+: N^+ \mH\to E_0^+ \mH$; the proofs
  for the other three maps are similar.
  Let $h\in N^+\mH$ and consider $E_0^- h$. For the map 
  $N^-: E_0^- \mH\to N^- \mH$, we have proved above that
  $$
    \|E_0^- h\|\lesssim \|N^-(E_0^- h)\|.
  $$
  This yields $\|h\|-\|E_0^+ h\|\lesssim \|N^-(E_0^+ h)\|\lesssim \|E_0^+ h\|$, by the reverse triangle inequality and since $E_0^-h= h-E_0^+h$ and $N^- h=0$.
  This estimate proves that $E_0^+: N^+ \mH\to E_0^+ \mH$
  is an injective semi-Fredholm operator.
  
(ii)
We note that the projections $\chi^+(D_0B)$ depend continuously (in fact, analytically) on $B$.
  This follows from quadratic estimates as in \cite[Sec.~6]{AKMc}, and
  allows us to apply perturbation theory for Fredholm operators.
  We reduce to fixed spaces as in \cite[Lem.~4.3]{AAMc2}.
  To show surjectivity of the restricted projections, let $\widetilde B$
  be the transform of self-adjoint coefficients 
  $\tilde A= \begin{bmatrix} a & b \\ c & d \end{bmatrix}$.
  Consider self-adjoint coefficients 
  $\tilde A_t= \begin{bmatrix} a & tb \\ tc & d \end{bmatrix}$
  for $t\in [0,1]$, with transform $\widetilde B_t$.
  For $t=0$, we have explicit inverses for the restricted projections.
  To see this, we write $N^\pm= (I\pm N)/2$, with $N$ as in the proof of 
  Lemma~\ref{lem:rellich}.
  We always have $ND_0+D_0N=0$, but for $B=\widetilde B_0$
  we also have $N\widetilde B_0= \widetilde B_0 N$.
  As a consequence $N E_0^\pm =E_0^\mp N$ when 
  $E_0^\pm= \chi^\pm(D_0\widetilde B_0)$.
  For example, consider $h\in E^+_0\mH$.
  Then 
  $$
  E_0^+N^+ h= \tfrac 12 (E_0^+ + E_0^+ N h)= \tfrac 12(h+N E_0^-h)
    = \tfrac{1}{2}h. 
  $$ 
  This shows that, modulo a factor $2$, the restricted projections
  $E_0^+: N^+ \mH\to E_0^+ \mH$ and
  $N^+: E_0^+ \mH\to N^+ \mH$ are inverse maps.
  Similarly, the other six restricted projections among 
  $N^\pm|_{E_0^\pm\mH}$
  and $E_0^\pm|_{N^\pm\mH}$ are seen to be pairwise
  inverse, modulo a factor $2$.

  By stability of the index for Fredholm operators, it follows
  that all eight restricted projections are invertible also for
  $\widetilde B= \widetilde B_1$.
  By the continuity of $B\mapsto \chi^\pm(D_0 B)$, this continues to
  hold for all non-selfadjoint $B$ such that $\|B-\widetilde B\|_\infty<\epsilon$.
  
(iii) 
   Finally, we prove the result for $\mH^{-1}$.
   By perturbation theory as in (ii), it suffices to show that the eight restricted 
   projections are injective semi-Fredholm maps in the $\mH^{-1}$
   topology, for coefficients $\widetilde B$.
  Given $h^+\in E_0^\pm \mH^{-1}$, we write
  $h^+= D_0\widetilde B \tilde h^+$ with 
  $\tilde h^+\in E_0^+ \mH$.
  From Lemma~\ref{lem:rellich} we obtain the estimate
  $\|\tilde h^+\|_{\rev L^2}\lesssim \|(\widetilde B\tilde h^+)_\no\|_{\rev L^2}$.
  By Lemma~\ref{lem:HtoHminone}, we equivalently have
  $\|h^+\|_{\rev W^{-1,2}}\lesssim \|(h^+)_\ta\|_{\rev W^{-1,2}}$
  since $D_0(\widetilde B\tilde h^+)_\no= (h^+)_\ta$, as  
  $D_0$ swaps normal and tangential parts.
  We also have the estimate
  $\|h^+\|_{\rev W^{-1,2}}\simeq\|\tilde h^+\|_{\rev L^2}\lesssim \|(\widetilde B\tilde h^+)_\ta\|_{\rev L^2}  \eqsim \|(h^+)_\no\|_{\rev W^{-1,2}}$.
  Given these two $\mH^{-1}$ Rellich estimates, invertibility of the 
  eight restricted projections follows as in the space $\mH$.
\end{proof}

\begin{proof}[Proof of Theorem~\ref{thm:dirneu}]
(i) We first establish the Neumann solvability estimate
$\|\nabla u\|_\mX\lesssim
\| \pd_{\nu_{A_\rho}}u_\rho|_M\|_{L^2(M,\omega_0^{-1})}$.
 It suffices 
  to show that $\nabla u\mapsto (f_0)_\no$ is an injective semi-Fredholm
  operator, as a bounded map from the closed subspace of $\mX$ consisting
  of gradients of solutions, to $\rev L^2$, which we prove next.
  Theorem~\ref{thm:N} shows that
  $$
  \|\nabla u\|_\mX\lesssim \|E_0^+ f_0\|+ \|E_0^0 f_0\|.
  $$
  Proposition~\ref{prop:restrictproj} shows that
  $N^-: E_0^+ \mH\to N^-\mH$ is an isomorphism, and hence
  $$
    \|E_0^+ f_0\|\lesssim \|(E_0^+ f_0)_\no\|.
  $$
  From \eqref{eq:postintegration}, where $h^+= E_0^+ f_0$,
  it now follows that 
  $$
    \|(h^+)_\no\|\le  \|(f_0)_\no\|
    + \epsilon \|\eta_t f_t\|_\mX + \|K f_t\|_\mX,
  $$
  with a compact operator $K$ and $\epsilon>0$ small depending on
  $\|A_\rho/\omega_\rho-A_0/\omega_0\|_*$.
  Combining the above three estimates, we obtain 
  $(1-C\epsilon)\|\nabla u\|_\mX\lesssim \|(f_0)_\no\|+  \|K(\nabla u)\|_\mX$,
  which proves that $\nabla u\mapsto (f_0)_\no$ has closed range, provided that $\epsilon<1/C$.
  Injectivity is immediate from \eqref{eq:Green}.
  
(ii)
The proof of the Dirichlet regularity solvability estimate
$$\|\nabla u\|_\mX\lesssim
 \|\nabla_M u_\rho|_M\|_{L^2(TM,\omega_0)}+ 
  \sum_i \Big|\int_{M_i}\pd_{\nu_{A_\rho}}u_\rho d\mu\Big|$$
 is similar.
 It suffices 
  to show that $\nabla u\mapsto (f_0)_\ta+ E_0^0 f_0$ 
  is an injective semi-Fredholm
  operator, as a bounded map from the closed subspace of $\mX$ consisting
  of gradients of solutions, to $\rev L^2$.
  That the range is closed follows as in (i), now using the invertibility of
  $N^+: E_0^+ \mH\to N^+\mH$ shown in
  Proposition~\ref{prop:restrictproj}.
  To see injectivity, we note that $(f_0)_\ta=0$ entails
  that $u_\rho$ is locally constant on $M$ and that $E_0^0 f_0=0$
  entails that
  $\int_{M_i}\pd_{\nu_{A_\rho}} u_\rho d\mu=0$ for each 
  connected component $M_i$ of $M$. 
  Therefore $\nabla u=0$ follows from \eqref{eq:Green}.
  
(iii)
To prove the Dirichlet solvability estimate
$$
\|\nabla u\|_\mY\lesssim  \|\nabla_M u_\rho|_M\|_{W^{-1,2}(TM,\omega_0)}
+  \sum_i \Big|\int_{M_i}\pd_{\nu_{A_\rho}}u_\rho d\mu\Big|,
$$
 it suffices 
  to show that $\nabla u\mapsto (f_0)_\ta+ E_0^0 f_0$ 
  is an injective semi-Fredholm
  operator, as a bounded map from the closed subspace of $\mY$ consisting
  of gradients of solutions, to $\rev W^{-1,2}$.
  Theorem~\ref{thm:D} shows that
  $\|\nabla u\|_\mY\lesssim \|E_0^+ f_0\|_{\rev W^{-1,2}}+ \|E_0^0 f_0\|$.
    We now use the invertibility of
  $N^+: E_0^+ \mH^{-1}\to N^+\mH^{-1}$, provided by
  Proposition~\ref{prop:restrictproj}.
  As in (ii), we conclude from \eqref{eq:postintegration} an estimate
  $(1-C\epsilon)\|\nabla u\|_\mY\lesssim \|(f_0)_\ta\|_{\rev W^{-1,2}}
  + \|K(\nabla u)\|_\mY$,
  with some compact operator $K$, from which it follows that 
  the range of $\nabla u\mapsto (f_0)_\ta+ E_0^0 f_0$ is closed.
  
  To prove injectivity, assume that $\div A \nabla u=0$
  with $\nabla u\in \mY$, $(f_0)_\ta=0$ and $E_0^0 f_0=0$.
  It suffices to show regularity $\nabla u\in\mX$.
  Indeed, as in (ii) it then follows from \eqref{eq:Green} that $\nabla u=0$.
  To show that $\nabla u\in\mX$, we use \eqref{eq:tracesplitrepr} 
  to write
  $$
    f_0 = h^++ h^-_\mY+ h^-_\mX,
  $$   
  where 
  \begin{align*}
    h^-_\mY &= S(\mE_t \eta_t f_t)|_{t=0},\\
    h^-_\mX &= \widetilde S(\eta'_t f_t)|_{t=0}.
  \end{align*}
  We have $h^-_\mY\in E_0^-\mH^{-1}$, and from part (iii)
  of the proof of Theorem~\ref{thm:N} it follows that 
  $h^-_\mX\in E_0^-\mH$ and that $\widetilde S(\eta'_t f_t)\in \mX$.
  Now note that $f_0\in N^-\mH^{-1}$, $h^+\in E_0^+\mH^{-1}$ 
  and that 
  $E_0^+: N^-\mH^{-1}\to E_0^+\mH^{-1}$
  and $E_0^-: N^-\mH^{-1}\to E_0^-\mH^{-1}$
  are isomorphisms by Proposition~\ref{prop:restrictproj}.
  Define the Hankel type operator 
  $$
    H: E_0^-\mH^{-1}\to E_0^+\mH^{-1}: h^-\to h^+
  $$
  if $f\in N^-\mH^{-1}$ with $h^\pm= E_0^\pm f$.
  We apply $H$ to $h^-_\mY+ h^-_\mX$, with $f=f_0$, to
  get
  $h^+= h^+_\mY+ h^+_\mX$ with
  $h^+_\mY= H(h^-_\mY)$ and $h^+_\mX= H(h^-_\mX)$.
  Rewrite \eqref{eq:postintegration} as
  $$
    (I -(I+ e^{-t\Lambda}H\gamma)S\mE_t)(\eta_t f_t)=
    e^{-t\Lambda} h^+_\mX+ \widetilde S(\eta'_t f_t), \qquad 0<t<\delta,
  $$
  where $\gamma$ denotes trace at $t=0$.
  Since the right hand side belongs to $\mX$ and the
  operator $(I+ e^{-t\Lambda}H\gamma)S\mE_t$ is bounded
  with small norm on both $\mY$ and $\mX$, using a Neumann series
  for the inverse it follows that $\nabla u\in \mX$ as required.
\end{proof}

\bibliographystyle{acm}

\end{document}